\documentclass[11pt]{article}

\usepackage[T1]{fontenc}
\usepackage[utf8]{inputenc}
\usepackage{lmodern}
\usepackage{amsmath,amssymb,amsthm,mathtools}
\usepackage{booktabs,array,longtable}
\usepackage{geometry}
\usepackage{xcolor}
\usepackage{listings}
\usepackage{enumitem}
\usepackage{hyperref}
\usepackage{cleveref}
\usepackage{microtype}

\hypersetup{
  colorlinks=true,
  linkcolor=blue!60!black,
  citecolor=blue!60!black,
  urlcolor=blue!60!black,
  pdftitle={Singular Cholesky Fibers over Finite Fields},
  pdfauthor={Hongfeng Wu and Kai Zhou},
  pdfsubject={},
  pdfkeywords={}
}
\allowdisplaybreaks

\newtheorem{theorem}{Theorem}[section]
\newtheorem{proposition}[theorem]{Proposition}
\newtheorem{lemma}[theorem]{Lemma}
\newtheorem{corollary}[theorem]{Corollary}
\theoremstyle{definition}
\newtheorem{definition}[theorem]{Definition}
\newtheorem{example}[theorem]{Example}

\theoremstyle{remark}
\newtheorem{remark}[theorem]{Remark}

\newcommand{\F}{\mathbb F}
\newcommand{\T}{\mathcal T}
\newcommand{\rank}{\operatorname{rank}}
\newcommand{\Null}{\operatorname{Null}}
\newcommand{\Col}{\operatorname{Col}}

\newcommand{\ind}{\mathbf 1}
\newcommand{\diag}{\operatorname{diag}}
\newcommand{\Span}{\operatorname{span}}
\newcommand{\Sym}{\operatorname{Sym}}

\lstdefinestyle{pythonstyle}{
  language=Python,
  basicstyle=\ttfamily\small,
  keywordstyle=\color{blue!65!black},
  stringstyle=\color{green!45!black},
  commentstyle=\color{gray!80!black},
  numbers=left,
  numberstyle=\scriptsize\color{gray},
  stepnumber=1,
  numbersep=8pt,
  showstringspaces=false,
  breaklines=true,
  frame=single,
  columns=fullflexible,
  keepspaces=true
}

\title{Singular Cholesky Fibers over Finite Fields}
\author{Hongfeng Wu$^{1}$ and Kai Zhou$^{2}$\\
\footnote{E-Mail addresses:
whfmath@gmail.com (H. Wu), kzhou@cslg.edu.cn (K. Zhou)}
{\small $^{1}$College of Science, North China University of Technology, Beijing, China}\\
{\small $^{2}$School of Mathematics and Statistics,}\\
{\small Suzhou University of Technology, Changshu, China}}
\date{}

\begin{document}

\maketitle

\begin{abstract}
For a finite field $\F_q$, consider the triangular Cholesky map
$\Gamma_{n,q}(U)=U^TU$ from upper triangular matrices to symmetric
matrices.  Generalized Cholesky theory describes the regular locus on
which all leading principal minors are nonzero, but it does not determine
the multiplicities or root ranks in a fiber over a singular target.  We
develop a fixed-target fiber theory for this singular boundary and answer
three questions posed by Cooper and Whitlatch.

Our principal results concern the zero fiber.  Over $\F_2$ we construct an
explicit, invertible, rank-preserving recursive bijection between
square-zero upper triangular matrices and upper triangular matrices
satisfying $U^TU=0$.  Over every finite field we determine the entire rank
distribution of this zero fiber: in even characteristic the square-zero
recurrence, and hence the rank-refined equinumerosity, persists over every
$\F_{2^e}$, whereas in odd characteristic its failure is measured by an
explicit quadratic-character correction.  We place these results in a
uniform framework by proving an exact first-pivot recursion for the fiber
cardinality over an arbitrary symmetric target.  Its regular
specialization gives the constant fiber sizes on leading-principal-minor
cones, while its zero-pivot branch explains the new singular behavior.
For binary diagonal targets we further compress the rank-refined count to
$O(n^2)$ integer-arithmetic transitions and show that it depends on the
order, not merely the number, of the diagonal entries.  Executable
implementations and exhaustive low-order checks are collected in an
appendix.
\end{abstract}

\section{Introduction}
For a finite field $\F_q$, let $\T_n(\F_q)$ be the set of all
$n\times n$ upper triangular matrices and let $\Sym_n(\F_q)$ be the
space of symmetric $n\times n$ matrices.  The map
\[
 \Gamma_{n,q}:\T_n(\F_q)\longrightarrow\Sym_n(\F_q),
 \qquad \Gamma_{n,q}(U)=U^TU,
\]
will be called the triangular Cholesky map.  Summing the cardinalities
of all its fibers merely recovers $|\T_n(\F_q)|$.  The nontrivial
enumerative problem is to determine the individual fiber
$\Gamma_{n,q}^{-1}(A)$ for a prescribed symmetric target $A$ and, when
the fiber is nonempty, to refine its cardinality by the rank of the
root.

The regular part of this problem is governed by triangular elimination.
Cooper, Hanna, and Whitlatch developed a finite-field analogue of
positive definiteness in terms of nonzero square leading principal
minors \cite{CooperHannaWhitlatch2024}.  Khare and Vishwakarma introduced
the larger leading-principal-minor (LPM) cones and their generalized
Cholesky geometry over real and complex fields
\cite{KhareVishwakarma2025}.  Vishwakarma subsequently developed the
finite-field theory \cite{Vishwakarma2027}, considering symmetric
matrices whose leading principal minors are all nonzero and have a
prescribed quadratic-character pattern.  His generalized factorization
$A=LA_{\epsilon}L^T$, with $L$ invertible lower triangular and
$A_{\epsilon}$ a fixed representative, parametrizes these regular
cones over definite fields and in characteristic two; the cardinalities
of the cones are determined over all finite fields.  That work also
develops Frobenius-compatible maps and transported group structures.
In the special case $A_{\epsilon}=I_n$, transposition identifies this
construction with the restriction of $\Gamma_{n,q}$ to invertible upper
triangular matrices.  Very recently, Ayyer and Prasad compared several
inequivalent notions of positive definiteness and positive
semidefiniteness over finite fields \cite{AyyerPrasad2026}.  Their type-2
positive-semidefinite class consists of targets $A=LL^T$ with $L$ lower
triangular and with nonnegative diagonal entries, and they obtain general
upper bounds for the number of such targets.

These target-space results and the present fixed-target problem are
complementary but logically distinct.  Enumerating or parametrizing the
image of a triangular factorization does not determine the cardinality of
$\Gamma_{n,q}^{-1}(A)$ for a prescribed $A$, still less the distribution
of the ranks of its roots.  This distinction is decisive on the singular
boundary.  If a leading pivot vanishes, triangular elimination is no
longer deterministic: the fiber may be empty, or it may branch over
rank-one perturbations, and different branches may contain roots of
different ranks.  The zero matrix is the most singular target.  Our
first-pivot recursion makes the interface precise: it recovers the
expected constant fiber sizes on the regular LPM locus and then continues
through every zero-pivot branch.

The square-zero family on the other side of our main bijection has a
classical enumerative history; see Kirillov \cite{Kirillov1995} and the
explicit formula of Ekhad and Zeilberger \cite{EkhadZeilberger1996}.
Our immediate starting point is the rank-refined work of Cooper and
Whitlatch \cite{CooperWhitlatch2025}.  Over $\F_2$ they considered
\[
\begin{aligned}
B_n(r)&=\{X\in\T_n(\F_2):\rank X=r,\ X^2=0\},\\
C_n(r)&=\{U\in\T_n(\F_2):\rank U=r,\ U^TU=0\},
\end{aligned}
\]
and proved the rank-refined equinumerosity
\[
 |B_n(r)|=|C_n(r)|
\]
by showing that both sides satisfy the recurrence
\[
 f_{n,0}=1,
 \qquad
 f_{n,r}=2^r f_{n-1,r}
 +\bigl(2^{n-r}-2^{r-1}\bigr)f_{n-1,r-1}
\]
for $1\le r\le\lfloor n/2\rfloor$, with $f_{n,r}=0$ outside the
feasible rank range.  Here $f_{n,r}$ denotes the common value of
$|B_n(r)|$ and $|C_n(r)|$. This work is related to successful
pressing sequences for bicolored graphs and binary matrix elimination
\cite{CooperDavis2016}.  Cooper and Whitlatch posed three further
questions: determine the Cholesky roots of an arbitrary fixed binary
matrix; construct an explicit rank-preserving bijection between
$B_n(r)$ and $C_n(r)$; and extend the zero-target enumeration to other
finite fields.

Our two principal results answer the latter two questions.  First,
\Cref{thm:explicit-bijection} constructs, recursively and without any
cardinality-matching argument, an explicit invertible map
\[
 \Phi_n:B_n(r)\longrightarrow C_n(r)
\]
that preserves rank for every $n$ and $r$.  Second,
\Cref{thm:even-recurrence,thm:odd-recurrence} determine the rank
distribution of the zero fiber over every finite field.  The
rank-refined square-zero recurrence persists over every $\F_{2^e}$.  If
$q$ is odd, and $\chi$ denotes the quadratic character, the extension
coefficient acquires the explicit correction
\[
 \kappa_n(q)=
 \begin{cases}
 (q-1)\chi((-1)^{n/2})q^{(n-2)/2},&n\text{ even},\\
 0,&n\text{ odd},
 \end{cases}
\]
which arises from the discriminant of a nondegenerate quotient quadratic
space and explains why the equinumerosity already fails in order two.

These central theorems are supported by a broader fixed-target framework.
For every finite field $\F_q$ and every symmetric target $A$,
\Cref{thm:general-recursion} gives an exact first-pivot recursion for
$|\Gamma_{n,q}^{-1}(A)|$.  Its regular specialization
\Cref{cor:regular-fibers} shows that, over an odd field, the standard
Cholesky map has $2^n$ roots on the all-square LPM cone and no roots on
the other regular sign cones, whereas over an even field it has one root
at every regular target.  Finally, an upper triangular target in the
image must be diagonal, and \Cref{thm:diagonal-recurrence} gives a
rank-refined recursion for every binary diagonal target using only
$O(n^2)$ integer-arithmetic transitions.  The count depends on the order
of the diagonal entries, not merely on their number.

The paper is organized as follows.  \Cref{sec:preliminaries} fixes the
notation and records triangular-congruence invariance.
\Cref{sec:general-recursion} proves the arbitrary-field first-pivot
recursion and identifies its regular locus.  \Cref{sec:diagonal}
develops the binary diagonal state recursion.  \Cref{sec:bijection}
constructs the explicit rank-preserving bijection, and
\Cref{sec:zero-fibers} treats zero fibers over arbitrary finite fields.
Low-order verification is recorded in \Cref{sec:verification}; all
implementations are placed in \Cref{app:algorithms}.

\section{Preliminaries}
\label{sec:preliminaries}

\begin{definition}
Let $A\in M_n(\F_q)$. If there exists an upper triangular matrix
$U\in\T_n(\F_q)$ such that
\[
U^TU=A,
\]
then $U$ is called an upper triangular Cholesky root of $A$, and the
equality $A=U^TU$ is called a Cholesky decomposition.
\end{definition}

\begin{remark}
Over any finite field, $U^TU$ is symmetric. Thus a non-symmetric matrix
has no Cholesky roots in this sense.
\end{remark}

\begin{definition}
For $A\in\Sym_n(\F_q)$, define
\[
 \nu_q(A)=\#\{U\in\T_n(\F_q):U^TU=A\}.
\]
For binary targets we abbreviate $\nu_2(A)$ to $\nu(A)$.  For every
finite field we use the empty-matrix convention
\[
 \nu_q(\varnothing)=1.
\]
\end{definition}

For $a\in\F_q$, put
\[
 \rho_q(a)=\#\{u\in\F_q:u^2=a\}.
\]
The value $\rho_q(0)$ is always $1$.  If $q$ is even, the Frobenius
map $u\mapsto u^2$ is an automorphism of $\F_q$, so $\rho_q(a)=1$ for
every $a$.  If $q$ is odd and $\chi$ denotes the quadratic character
of $\F_q$, extended by $\chi(0)=0$, then
\[
 \rho_q(a)=1+\chi(a)\qquad(a\in\F_q).
\]
Indeed, a nonzero square has exactly the two roots $u$ and $-u$, while
a nonsquare has none.
\begin{proposition}
\label{prop:congruence}
If $P\in\T_n(\F_q)$ is invertible, then
\[
\nu_q(P^TAP)=\nu_q(A)
\]
for every $A\in\Sym_n(\F_q)$.
\end{proposition}

\begin{proof}
Let
\[
\mathcal S_A=\{U\in\T_n(\F_q):U^TU=A\}.
\]
Define the map
\[
\Phi:\mathcal S_A\to\mathcal S_{P^TAP},\qquad U\mapsto UP.
\]
We first verify that $\Phi$ is well defined. Since $U$ and $P$ are both upper triangular, their product $UP$ is upper triangular. Moreover,
\[
(UP)^T(UP)=P^T U^T U P=P^T A P,
\]
so indeed $UP\in\mathcal S_{P^TAP}$.

To show that $\Phi$ is a bijection, define
\[
\Psi:\mathcal S_{P^TAP}\to\mathcal S_A,\qquad V\mapsto VP^{-1}.
\]
Because $P$ is invertible and upper triangular, its inverse $P^{-1}$ is also upper triangular; hence $VP^{-1}$ is upper triangular for every upper triangular $V$. Furthermore,
\[
(VP^{-1})^T(VP^{-1})
=(P^{-1})^T V^T V P^{-1}
=(P^{-1})^T P^T A P P^{-1}
=(P P^{-1})^T A (P P^{-1})
=A.
\]
Thus $\Psi(V)\in\mathcal S_A$, so $\Psi$ is well defined.

Finally, for $U\in\mathcal S_A$ and $V\in\mathcal S_{P^TAP}$,
\[
\Psi(\Phi(U))=(UP)P^{-1}=U,
\qquad
\Phi(\Psi(V))=(VP^{-1})P=V.
\]
Therefore $\Phi$ is a bijection between $\mathcal S_A$ and $\mathcal S_{P^TAP}$, which implies
\[
\nu_q(P^TAP)=|\mathcal S_{P^TAP}|=|\mathcal S_A|=\nu_q(A).
\]
\end{proof}
\begin{remark}
\Cref{prop:congruence} shows that the Cholesky fiber cardinality is an
invariant of the congruence action of the invertible upper triangular
group.  Equivalently, after transposition this is the lower triangular
congruence action used to parametrize the regular
leading-principal-minor cones in \cite{Vishwakarma2027}.  On the
singular boundary the action still preserves fiber sizes, although it
need not be transitive.
\end{remark}

\section{A Complete First-Pivot Recursion over Finite Fields}
\label{sec:general-recursion}

\subsection{First-Row Block Decomposition}

Write
\[
A=
\begin{pmatrix}
a&b^T\\
b&D
\end{pmatrix},
\qquad
a\in\F_q,\quad b\in\F_q^{n-1},\quad
D\in\Sym_{n-1}(\F_q).
\]
Every $U\in\T_n(\F_q)$ can be written uniquely as
\[
U=
\begin{pmatrix}
u&x^T\\
0&V
\end{pmatrix},
\qquad
u\in\F_q,\quad x\in\F_q^{n-1},\quad
V\in\T_{n-1}(\F_q).
\]
A direct computation gives
\[
U^TU
=
\begin{pmatrix}
u^2&ux^T\\
ux&xx^T+V^TV
\end{pmatrix}.
\]
The block identities above give an arbitrary-field recursion.  The
coefficient $\rho_q(a)$ records exactly the possible choices of the
first diagonal entry of the root.
\begin{theorem}
\label{thm:general-recursion}
With the notation above,
\[
\boxed{
\nu_q(A)=
\begin{cases}
\rho_q(a)\,\nu_q(D-a^{-1}bb^T),&a\ne0,\\[2mm]
0,&a=0,\ b\ne0,\\[2mm]
\displaystyle\sum_{x\in\F_q^{n-1}}\nu_q(D-xx^T),
&a=0,\ b=0.
\end{cases}}
\]
\end{theorem}

\begin{proof}
Equating the three blocks of $U^TU$ with the corresponding blocks of
$A$ gives
\begin{equation}
 u^2=a,\qquad ux=b,\qquad V^TV=D-xx^T.
 \label{eq:first-pivot-system}
\end{equation}
We count the solutions of this system in three disjoint cases.

\paragraph{Case 1: $a\ne0$.}
Every solution $u$ of $u^2=a$ is nonzero.  For such a fixed $u$, the
middle equation in \eqref{eq:first-pivot-system} has the unique solution
\[
 x=u^{-1}b.
\]
Since $u^2=a$, we have $u^{-2}=a^{-1}$, and hence
\[
 xx^T=(u^{-1}b)(u^{-1}b)^T=u^{-2}bb^T=a^{-1}bb^T.
\]
The last equation in \eqref{eq:first-pivot-system} therefore becomes
\[
 V^TV=D-a^{-1}bb^T.
\]
For each of the $\rho_q(a)$ possible values of $u$, there are exactly
$\nu_q(D-a^{-1}bb^T)$ possible matrices $V$.  The value of $x$ is then
forced, and distinct choices of $u$ give distinct matrices $U$ because
they have different $(1,1)$-entries.  Consequently
\[
 \nu_q(A)=\rho_q(a)\nu_q(D-a^{-1}bb^T).
\]
This argument also covers the possibility $\rho_q(a)=0$, in which case
the fiber is empty.

\paragraph{Case 2: $a=0$ and $b\ne0$.}
Because $\F_q$ is a field, the equation $u^2=0$ forces $u=0$.  The
middle equation in \eqref{eq:first-pivot-system} would then read
$0=b$, contradicting $b\ne0$.  Thus no Cholesky root exists and
$\nu_q(A)=0$.

\paragraph{Case 3: $a=0$ and $b=0$.}
Again $u=0$, but now the equation $ux=b$ is satisfied by every
$x\in\F_q^{n-1}$.  Once $x$ is fixed, the only remaining condition is
\[
 V^TV=D-xx^T,
\]
which has $\nu_q(D-xx^T)$ solutions $V$.  The block decomposition of
$U$ is unique, so solution sets belonging to different choices of $x$
are disjoint.  Summing their cardinalities gives
\[
 \nu_q(A)=\sum_{x\in\F_q^{n-1}}\nu_q(D-xx^T).
\]
The three cases exhaust all possibilities for $(a,b)$ and prove the
recursion.
\end{proof}
\begin{corollary}
\label{cor:binary-recursion}
For
\[
 A=\begin{pmatrix}a&b^T\\b&D\end{pmatrix}\in\Sym_n(\F_2),
\]
the recursion becomes
\[
 \nu(A)=
 \begin{cases}
  \nu(D+bb^T),&a=1,\\[1mm]
  0,&a=0,\ b\ne0,\\[1mm]
  \displaystyle\sum_{x\in\F_2^{n-1}}\nu(D+xx^T),&a=0,\ b=0.
 \end{cases}
\]
\end{corollary}

\begin{proof}
The square map on $\F_2$ is the identity, so $\rho_2(1)=1$ and $1$ is
the only nonzero value of $a$.  Moreover subtraction and addition agree
in characteristic two.  Substituting these facts into
\Cref{thm:general-recursion} gives the stated formulas.
\end{proof}

\begin{corollary}
\label{cor:existence}
A symmetric matrix $A$ over a finite field has an upper triangular
Cholesky root if and only if at least one branch produced by repeated
application of \Cref{thm:general-recursion} reaches the empty $0\times 0$ target matrix.
\end{corollary}

\begin{proof}
We prove the equivalence by interpreting the recursion in
Theorem~\ref{thm:general-recursion} as a branching process that
partitions the set of Cholesky roots into disjoint classes.

At a generic step, the target matrix is written as
\[
A=\begin{pmatrix} a & b^T \\ b & D \end{pmatrix},
\]
and every candidate root is written as
\[
U=\begin{pmatrix} u & x^T \\ 0 & V \end{pmatrix}.
\]
The proof of Theorem~\ref{thm:general-recursion} shows that the
condition $U^TU=A$ is equivalent to the system
\[
u^2=a,\qquad ux=b,\qquad V^TV=D-xx^T.
\]
Thus each valid choice of the first diagonal entry $u$ and, when
necessary, of the vector $x$, selects a disjoint subset of the fiber.
The remaining block $V$ must then be a Cholesky root of a smaller
symmetric matrix.

Consider the recursion tree whose nodes are labeled by the target
matrices that appear during this process. The root node is $A$. At
a node labeled by
\[
\begin{pmatrix} a & b^T \\ b & D \end{pmatrix},
\]
exactly one of the following occurs:
\begin{enumerate}[label=\textup{(\roman*)}]
\item If $a\ne0$, then the node has exactly one child, labeled by
      $D-a^{-1}bb^T$, and the contribution from this node to the
      parent is multiplied by the nonnegative integer
      $\rho_q(a)=\#\{u\in\F_q:u^2=a\}$.
\item If $a=0$ and $b\ne0$, then the node is a dead end: the
      recursion returns $0$, and no branch continues from it.
\item If $a=0$ and $b=0$, then for every $x\in\F_q^{n-1}$ the
      node has a child labeled by $D-xx^T$. The contributions of
      these children are added, and the choice of $x$ keeps the
      resulting subsets disjoint.
\end{enumerate}
The recursion terminates precisely when the target matrix becomes the
empty $0\times 0$ matrix. By convention,
\[
\nu_q(\varnothing)=1.
\]
A branch that reaches the empty target therefore contributes a positive
integer to the total count: it is the product of the positive numbers
$\rho_q(a)$ encountered at the nonzero-pivot nodes along that branch.
A branch that reaches a dead end of type (ii) contributes zero.

Consequently, the total number $\nu_q(A)$ is the sum of the
contributions of all branches that reach the empty target. Since every
contribution of such a successful branch is a positive integer, the sum
is positive if and only if at least one successful branch exists.

It remains to connect this combinatorial condition to the existence of
a root. If $U\in\T_n(\F_q)$ is a Cholesky root of $A$, then at each
recursive step the actual first diagonal entry $u$ of the appropriate
block of $U$ satisfies $u^2=a$. Hence at every nonzero-pivot node
the corresponding value $\rho_q(a)$ is at least $1$, because $u$
itself is a solution. When the first row is zero, the actual vector
$x$ selects the child through which the recursion continues. Thus
following the root $U$ through the recursion yields a branch that
never reaches a dead end of type (ii) and eventually arrives at the
empty matrix. Therefore the existence of a root implies the existence
of a successful branch.

Conversely, if a successful branch exists, then tracing it backwards
constructs a Cholesky root: at each nonzero-pivot node, choose one of
the $\rho_q(a)>0$ solutions $u$, set $x=u^{-1}b$, and recursively
construct the remaining block $V$; at each zero-pivot node with zero
first row, use the $x$ that labels the branch. The terminal node is
the empty matrix, which has the empty Cholesky root by convention.
Assembling the blocks gives a matrix $U\in\T_n(\F_q)$ with
$U^TU=A$.

Hence $A$ has an upper triangular Cholesky root if and only if the
recursion in Theorem~\ref{thm:general-recursion} has at least one branch
that reaches the empty target.
\end{proof}
We next identify the regular part of the Cholesky map.  For $A\in\Sym_n(\F_q)$, let $A_{[k],[k]}$ denote the $k\times k$ leading principal submatrix of $A$ (i.e. the submatrix formed by the first $k$ rows and columns, with indices in $[k]$). Define
\[
 \Delta_k(A)=\det A_{[k],[k]}\quad(1\le k\le n),\qquad \Delta_0(A)=1,
\]
where $[k]=\{1,\ldots,k\}$.  Call $A$ an \emph{LPM matrix} if
$\Delta_k(A)\ne0$ for every $k$.

\begin{corollary}
\label{cor:regular-fibers}
Let $A\in\Sym_n(\F_q)$ be an LPM matrix and put
\[
 \delta_k(A)=\frac{\Delta_k(A)}{\Delta_{k-1}(A)}
 \qquad(1\le k\le n).
\]
Then
\begin{equation}
 \nu_q(A)=\prod_{k=1}^n\rho_q\bigl(\delta_k(A)\bigr).
 \label{eq:regular-fiber-product}
\end{equation}
Consequently,
\[
 \nu_q(A)=
 \begin{cases}
  1,&q\text{ even},\\[1mm]
  2^n,&q\text{ odd and }\chi(\Delta_k(A))=1
       \text{ for every }k,\\[1mm]
  0,&q\text{ odd and }\chi(\Delta_k(A))=-1
       \text{ for at least one }k.
 \end{cases}
\]
\end{corollary}

\begin{proof}
We prove the product formula \eqref{eq:regular-fiber-product} by induction on $n$. For $n=0$, the empty product is $1$, and $\nu_q(\varnothing)=1$, so the formula holds.

Assume $n\ge 1$ and write
\[
 A=\begin{pmatrix}a&b^T\\b&D\end{pmatrix},
\]
where $a\in\F_q$, $b\in\F_q^{n-1}$, and $D\in\Sym_{n-1}(\F_q)$. Since $A$ is an LPM matrix, its first leading principal minor $\Delta_1(A)=a$ is nonzero. Hence $a\ne0$, and by Theorem~\ref{thm:general-recursion},
\[
 \nu_q(A)=\rho_q(a)\,\nu_q(S),
\]
where
\[
 S=D-a^{-1}bb^T.
\]

We claim that $S$ is also an LPM matrix. For $1\le j\le n-1$, let $b_j$ denote the vector consisting of the first $j$ entries of $b$, and write $D_j=D_{[j],[j]}$. The leading $(j+1)\times(j+1)$ block of $A$ is
\[
 A_{j+1}=\begin{pmatrix}a&b_j^T\\b_j&D_j\end{pmatrix}.
\]
Multiplying on the left and right by determinant-one upper and lower triangular matrices gives
\[
 \begin{pmatrix}1&0\\-a^{-1}b_j&I_j\end{pmatrix}
 A_{j+1}
 \begin{pmatrix}1&-a^{-1}b_j^T\\0&I_j\end{pmatrix}
 =
 \begin{pmatrix}a&0\\0&D_j-a^{-1}b_jb_j^T\end{pmatrix}.
\]
Taking determinants yields
\begin{equation}\label{eq:schur-leading-minors}
 \Delta_{j+1}(A)
 =a\det(D_j-a^{-1}b_jb_j^T)
 =a\,\Delta_j(S).
 \end{equation}
Because $A$ is an LPM matrix, $\Delta_{j+1}(A)\ne0$ for every $j$. Since $a\ne0$, it follows from \eqref{eq:schur-leading-minors} that $\Delta_j(S)\ne0$ for every $j=1,\ldots,n-1$. Thus $S$ is an LPM matrix of order $n-1$.

By the induction hypothesis,
\[
 \nu_q(S)=\prod_{j=1}^{n-1}
   \rho_q\!\left(\frac{\Delta_j(S)}{\Delta_{j-1}(S)}\right).
\]
Therefore
\[
 \nu_q(A)=\rho_q(a)\prod_{j=1}^{n-1}
   \rho_q\!\left(\frac{\Delta_j(S)}{\Delta_{j-1}(S)}\right).
\]
Here $a=\delta_1(A)$. It remains to identify the remaining factors with $\rho_q(\delta_k(A))$ for $k=2,\ldots,n$.

From \eqref{eq:schur-leading-minors} we have
\[
 \Delta_j(S)=\frac{\Delta_{j+1}(A)}{a},\qquad j=1,\ldots,n-1.
\]
In particular, for $j=1$,
\[
 \Delta_1(S)=\frac{\Delta_2(A)}{a}
 =\frac{\Delta_2(A)}{\Delta_1(A)}
 =\delta_2(A),
\]
so the first factor in the product over $j$ is
\[
 \rho_q\!\left(\frac{\Delta_1(S)}{\Delta_0(S)}\right)
 =\rho_q(\Delta_1(S))
 =\rho_q(\delta_2(A)).
\]
For $j\ge2$, using \eqref{eq:schur-leading-minors} and the definition of $\delta_k$, we get
\[
 \frac{\Delta_j(S)}{\Delta_{j-1}(S)}
 =
 \frac{\Delta_{j+1}(A)/a}{\Delta_j(A)/a}
 =
 \frac{\Delta_{j+1}(A)}{\Delta_j(A)}
 =
 \delta_{j+1}(A).
\]
Hence the product over $j=1,\ldots,n-1$ becomes
\[
 \prod_{j=1}^{n-1}
 \rho_q\!\left(\frac{\Delta_j(S)}{\Delta_{j-1}(S)}\right)
 =
 \prod_{k=2}^{n}\rho_q(\delta_k(A)).
\]
Combining this with the initial factor $\rho_q(a)=\rho_q(\delta_1(A))$, we obtain
\[
 \nu_q(A)
 =
 \rho_q(\delta_1(A))\prod_{k=2}^{n}\rho_q(\delta_k(A))
 =
 \prod_{k=1}^{n}\rho_q(\delta_k(A)).
\]
This proves \eqref{eq:regular-fiber-product}.

Now we specialize to even and odd characteristics. If $q$ is even, then every nonzero element of $\F_q$ has exactly one square root, because the Frobenius map $x\mapsto x^2$ is bijective on $\F_q$. Hence $\rho_q(\alpha)=1$ for every $\alpha\ne0$, and the product equals $1$.

Suppose $q$ is odd. For $\alpha\ne0$, the equation $u^2=\alpha$ has two solutions if $\alpha$ is a square, and no solutions otherwise. Equivalently, $\rho_q(\alpha)=1+\chi(\alpha)$, where $\chi$ is the quadratic character. Thus each factor in the product is $2$ when $\delta_k(A)$ is a square and $0$ when it is a nonsquare.

It remains to relate the square classes of the ratios $\delta_k$ to those of the leading principal minors $\Delta_k$. We have
\[
 \Delta_k(A)=\prod_{i=1}^k \delta_i(A),
 \qquad
 \delta_k(A)=\frac{\Delta_k(A)}{\Delta_{k-1}(A)}.
\]
If every $\delta_k(A)$ is a square, then every $\Delta_k(A)$, being a product of squares, is also a square. Conversely, if every $\Delta_k(A)$ is a square, then every quotient $\delta_k(A)=\Delta_k(A)/\Delta_{k-1}(A)$ is a square because the quotient of two nonzero squares in a finite field is again a square. Therefore all $\delta_k(A)$ are squares if and only if all $\Delta_k(A)$ are squares.

Consequently, the product is $2^n$ exactly when every leading principal minor $\Delta_k(A)$ is a square, and it is $0$ if at least one $\Delta_k(A)$ is a nonsquare. This gives the stated classification.
\end{proof}
\begin{remark}
The all-square cone in Corollary $\ref{cor:regular-fibers}$ is the finite-field
positive-definite cone of \cite{CooperHannaWhitlatch2024}.  The larger
regular sign cones and their generalized factorizations
$A=LA_{\epsilon}L^T$ are studied in \cite{Vishwakarma2027}.  Thus
Corollary $\ref{cor:regular-fibers}$ is the fiber-size counterpart for the standard
map $U\mapsto U^TU$.  It also isolates the genuinely singular
phenomenon: when a pivot vanishes and the first row vanishes with it,
the recursion branches over the rank-one perturbations $xx^T$.  This
branching is absent on every LPM cone.
\end{remark}

\subsection{Computational form}
The recursion is exact but may branch exponentially when successive
leading rows vanish.  Memoization identifies repeated target matrices
and substantially reduces low-order computations.  An implementation
over prime fields is given in \Cref{app:general-code}; its binary
specialization reproduces Corollary $\ref{cor:binary-recursion}$.  The algorithm
does not assert polynomial worst-case complexity; orbit compression of
the singular branching states remains an open problem.

\section{Upper Triangular Targets: Rank-Refined Recursion
for the Diagonal Case}
\label{sec:diagonal}

\subsection{Why the Problem Reduces to Diagonal Matrices}

\begin{proposition}
\label{prop:upper-is-diagonal}
If an upper triangular matrix $A$ admits a Cholesky root, then $A$ must
be diagonal.
\end{proposition}

\begin{proof}
Suppose $A=U^TU$ for some upper triangular matrix $U$. Then
\[
A^T=(U^TU)^T=U^T(U^T)^T=U^TU=A,
\]
so $A$ is symmetric.

Now let $A=(a_{ij})_{1\le i,j\le n}$. Since $A$ is upper triangular,
we have
\[
a_{ij}=0 \quad\text{whenever } i>j.
\]
Since $A$ is symmetric, we also have
\[
a_{ij}=a_{ji} \quad\text{for all } i,j.
\]
If $i<j$, then $j>i$, so by upper triangularity $a_{ji}=0$. Hence
\[
a_{ij}=a_{ji}=0 \quad\text{whenever } i<j.
\]
Thus $a_{ij}=0$ for all $i\neq j$, so $A$ is diagonal.
\end{proof}
Thus the ``upper triangular target matrix'' problem in the original
paper is, strictly speaking, the problem of counting Cholesky roots of
arbitrary binary diagonal matrices.

\subsection{One-Step Extension Lemma}

Let
\[
D\in\Sym_m(\F_2),\qquad \rank D=s,
\]
and suppose
\[
V^TV=D,\qquad \rank V=r.
\]
Consider
\[
U=
\begin{pmatrix}
V&w\\
0&c
\end{pmatrix}
\]
satisfying
\[
U^TU=D\oplus[d],\qquad d\in\F_2.
\]
Expanding the condition gives
\[
V^Tw=0,\qquad w^Tw+c=d.
\]

\begin{lemma}
\label{lem:intersection}
Under the above assumptions,
\[
\dim\bigl(\Col V\cap\Null V^T\bigr)=r-s.
\]
\end{lemma}

\begin{proof}
We prove the lemma by constructing an explicit isomorphism between
$\Null D/\Null V$ and $\Col V\cap\Null V^T$.

First note that since $V^TV=D$, we have
\[
\Null V\subseteq\Null D.
\]
Indeed, if $x\in\Null V$, then $Vx=0$, and hence
\[
Dx=V^TVx=V^T(0)=0,
\]
so $x\in\Null D$.

Define the map
\[
\theta:\Null D\longrightarrow \Col V\cap\Null V^T,
\qquad x\longmapsto Vx.
\]
We verify that $\theta$ is well defined. If $x\in\Null D$, then
$Vx\in\Col V$ by the definition of the column space. Moreover,
\[
V^T(Vx)=V^TVx=Dx=0,
\]
so $Vx\in\Null V^T$. Thus $Vx\in\Col V\cap\Null V^T$, as claimed.
The map $\theta$ is linear because it is the restriction of the linear
map $x\mapsto Vx$ to the subspace $\Null D$.

We now show that $\theta$ is surjective. Let
$w\in\Col V\cap\Null V^T$. Since $w\in\Col V$, there exists
$x\in\F_2^m$ such that $w=Vx$. Since $w\in\Null V^T$, we have
\[
0=V^Tw=V^TVx=Dx,
\]
so $x\in\Null D$. Therefore every element of $\Col V\cap\Null V^T$ is
of the form $\theta(x)$ for some $x\in\Null D$, proving surjectivity.

Next we compute the kernel of $\theta$. By definition,
\[
\ker\theta=\{x\in\Null D: Vx=0\}.
\]
If $x\in\Null V$, then $Vx=0$, and by the observation above
$x\in\Null D$, so $x\in\ker\theta$. Conversely, if $x\in\ker\theta$,
then $Vx=0$, so $x\in\Null V$. Hence
\[
\ker\theta=\Null V.
\]

Now apply the rank--nullity theorem to the linear map $\theta$:
\[
\dim\Null D=\dim\ker\theta+\dim\operatorname{im}\theta.
\]
Since $\ker\theta=\Null V$ and
$\operatorname{im}\theta=\Col V\cap\Null V^T$, this becomes
\[
\dim\Null D
=\dim\Null V+\dim(\Col V\cap\Null V^T).
\]
Solving for the dimension of the intersection gives
\[
\dim(\Col V\cap\Null V^T)
=\dim\Null D-\dim\Null V.
\]
Finally, using $\dim\Null D=m-\rank D=m-s$ and
$\dim\Null V=m-\rank V=m-r$, we obtain
\[
\dim(\Col V\cap\Null V^T)
=(m-s)-(m-r)=r-s.
\]
This completes the proof.
\end{proof}
The following lemma counts the possible one-step extensions of a fixed root $V$ of rank $r$ and determines exactly when such an extension preserves the rank or increases it by one.
\begin{lemma}
\label{lem:extension}
Fix a root $V$ of rank $r$ as above.
\begin{enumerate}[label=\textup{(\roman*)}]
\item The number of possible vectors $w\in\Null V^T$ is $2^{m-r}$,
      and each $w$ uniquely determines $c=d+w^Tw$.
\item If $d=0$, exactly $2^{r-s}$ of these extensions preserve the rank
      $r$, and the remaining $2^{m-r}-2^{r-s}$ extensions have rank
      $r+1$.
\item If $d=1$, every extension has rank $r+1$.
\end{enumerate}
\end{lemma}

\begin{proof}
Let
\[
U=\begin{pmatrix} V & w \\ 0 & c \end{pmatrix},
\qquad V\in\T_m(\F_2),\quad w\in\F_2^m,\quad c\in\F_2.
\]
The condition
\[
U^TU=D\oplus[d]
\]
expands to the two equations
\[
V^Tw=0
\qquad\text{and}\qquad
w^Tw+c^2=d.
\]
Over $\F_2$ we have $c^2=c$, so the second equation is equivalent to
\[
c=d+w^Tw.
\]
Thus, for any fixed $w$ satisfying $V^Tw=0$, the entry $c$ is uniquely
determined by the formula above. Conversely, every pair $(w,c)$ with
$w\in\Null V^T$ and $c=d+w^Tw$ gives a valid extension. Hence the set
of all possible extensions is in bijection with $\Null V^T$. Since
$\dim\Null V^T=m-\rank V=m-r$, we obtain
\[
|\Null V^T|=2^{m-r}.
\]
This proves part (i).

To determine the rank of the extended matrix, we compare the column
space of $U$ with that of $V$. Identify $\F_2^m$ with the subspace of
$\F_2^{m+1}$ consisting of vectors whose last coordinate is $0$. Under
this identification, the column space of $V$ becomes the subspace
\[
\Col V\times\{0\}
=\left\{\begin{pmatrix} Vy\\0\end{pmatrix}:y\in\F_2^m\right\}.
\]
The last column of $U$ is
\[
\begin{pmatrix} w\\ c\end{pmatrix}.
\]
Therefore the rank of $U$ is the same as the rank of $V$ if and only if
the last column lies in the span of the previous columns, that is,
\[
\begin{pmatrix} w\\ c\end{pmatrix}
\in \Col V\times\{0\}.
\]
This is equivalent to the two conditions
\[
w\in\Col V
\qquad\text{and}\qquad
c=0.
\]
If these conditions fail, then the last column is linearly independent
of the previous columns, so the rank increases by exactly one:
\[
\rank U=\rank V+1.
\]

Now suppose $d=0$. For $w\in\Col V\cap\Null V^T$, write $w=Vx$ with
$x\in\F_2^m$. Since $w\in\Null V^T$,
\[
0=V^Tw=V^TVx=Dx,
\]
so $x\in\Null D$. Therefore
\[
w^Tw=x^TV^TVx=x^TDx=0.
\]
Consequently,
\[
c=d+w^Tw=0+0=0.
\]
Thus every $w\in\Col V\cap\Null V^T$ gives an extension with $c=0$ and
$w\in\Col V$, so the rank is preserved. By Lemma $\ref{lem:intersection}$,
\[
|\Col V\cap\Null V^T|=2^{r-s}.
\]
Hence exactly $2^{r-s}$ extensions preserve the rank $r$.

All remaining vectors $w\in\Null V^T\setminus\Col V$ satisfy
$w\notin\Col V$, so regardless of the value of $c$, the last column is
not in the column space of $V$, and therefore the rank increases to
$r+1$. There are
\[
2^{m-r}-2^{r-s}
\]
such extensions. This proves part (ii).

Finally suppose $d=1$. For any $w\in\Null V^T$, the entry $c$ is
determined by
\[
c=1+w^Tw.
\]
In particular, for $w\in\Col V\cap\Null V^T$, we have shown that
$w^Tw=0$, so
\[
c=1+0=1.
\]
Thus even if $w\in\Col V$, the last coordinate of the last column is
$1$, while every column of $V$ has last coordinate $0$. Therefore the
last column of $U$ can never lie in the column space of $V$. Hence
every extension increases the rank by exactly one:
\[
\rank U=\rank V+1=r+1.
\]
This proves part (iii) and completes the proof.
\end{proof}
\subsection{Recursion Theorem for Diagonal Matrices}

Let
\[
D_i=\diag(d_1,\ldots,d_i),\qquad d_j\in\{0,1\},
\]
and define
\[
s_i=\sum_{j=1}^i d_j,\qquad z_i=i-s_i.
\]
Here $s_i=\rank D_i$, and $z_i$ is the number of zero diagonal entries
among the first $i$ entries.

\begin{definition}
Let $h_i(t)$ be the number of $i\times i$ upper triangular matrices
$U$ such that
\[
U^TU=D_i,\qquad \rank U=s_i+t.
\]
We use the convention
\[
h_0(0)=1,\qquad h_0(t)=0\quad(t\ne0).
\]
\end{definition}
The following theorem gives a rank-refined recursion for the number of upper triangular Cholesky roots of an arbitrary binary diagonal matrix, using the excess rank as the state variable.
\begin{theorem}
\label{thm:diagonal-recurrence}
For every $i\ge0$,
\[
 h_i(t)=0\qquad\text{unless}\qquad
 0\le t\le\left\lfloor\frac{z_i}{2}\right\rfloor.
\]
For $i\ge1$, the nonzero states satisfy the following recurrences.

If $d_i=1$, then
\[
\boxed{
h_i(t)=2^{z_i-t}h_{i-1}(t).
}
\qquad
\left(0\le t\le\left\lfloor\frac{z_i}{2}\right\rfloor\right)
\]

If $d_i=0$, then
\[
 \boxed{h_i(0)=h_{i-1}(0),}
\]
and, for $1\le t\le\lfloor z_i/2\rfloor$,
\[
\boxed{
h_i(t)
=2^t h_{i-1}(t)
+\left(2^{z_i-t}-2^{t-1}\right)h_{i-1}(t-1).
}
\]
Finally,
\[
\boxed{
\nu(D_n)=\sum_{t=0}^{\lfloor z_n/2\rfloor}h_n(t).
}
\]
\end{theorem}

\begin{proof}
We prove both the support assertion and the recurrences by induction on
$i$, using Lemma $\ref{lem:extension}$.  At $i=0$ we have $z_0=0$ and, by
definition, $h_0(0)=1$ while every other state is zero.  Thus the
assertion holds at the initial stage.

Fix $i\ge1$ and assume all assertions at stage $i-1$.  Let
$V\in\T_{i-1}(\F_2)$ be an old root, that is,
\[
V^TV=D_{i-1},
\qquad
\rank V=r.
\]
We wish to determine all possible extensions
\[
U=\begin{pmatrix} V & w \\ 0 & c \end{pmatrix}
\]
such that
\[
U^TU=D_i.
\]
By the block computation in the one-step extension setup, this is
equivalent to
\[
V^Tw=0,\qquad w^Tw+c=d_i.
\]
According to Lemma $\ref{lem:extension}$, for a fixed old root $V$, the
number of possible vectors $w\in\Null V^T$ is $2^{(i-1)-r}$, and
each $w$ uniquely determines $c$. Thus the total number of
extensions of $V$ is $2^{(i-1)-r}$.

We now separate the two possible values of $d_i$.

\paragraph{Case 1: $d_i=1$.}
Then $s_i=s_{i-1}+1$ and $z_i=z_{i-1}$.  Let
$0\le t\le\lfloor z_i/2\rfloor$, and consider an old root $V$
of rank
\[
r=s_{i-1}+t=(s_i-1)+t.
\]
Such a root contributes to $h_{i-1}(t)$, since the excess rank of
$V$ relative to $D_{i-1}$ is $t$. By Lemma $\ref{lem:extension}(iii)$,
when $d_i=1$, every extension of $V$ increases the rank by exactly
one. Therefore the extended matrix $U$ has rank
\[
r+1=(s_i-1+t)+1=s_i+t.
\]
Thus the excess rank of $U$ relative to $D_i$ is again $t$. The
number of extensions of each such $V$ is
\[
2^{(i-1)-r}
=2^{(i-1)-(s_i-1+t)}
=2^{i-s_i-t}
=2^{z_i-t},
\]
since $z_i=i-s_i$. Hence all roots counted by $h_{i-1}(t)$ extend
to roots counted by $h_i(t)$, and the number of extensions is the
same for each old root. Therefore
\[
h_i(t)=2^{z_i-t}h_{i-1}(t).
\]
If $t>\lfloor z_i/2\rfloor$, then
$t>\lfloor z_{i-1}/2\rfloor$, so the induction hypothesis gives
$h_{i-1}(t)=0$.  Every extension preserves the excess rank in this
case, and therefore $h_i(t)=0$ as well.  This proves both the formula
and the support assertion when $d_i=1$.

\paragraph{Case 2: $d_i=0$.}
Then $s_i=s_{i-1}$ and $z_i=z_{i-1}+1$.  We first treat $t=0$.
Let $V\in\T_{i-1}(\F_2)$ be an old root with $V^TV=D_{i-1}$, and consider an extension
\[
U=\begin{pmatrix} V & w\\ 0 & c\end{pmatrix}.
\]
Since $U$ is obtained from $V$ by adding a single column, its rank can increase by at most one:
\[
\rank V \le \rank U \le \rank V+1.
\]
Thus if $\rank U=s_i$, then $\rank V$ must be either $s_i$ or $s_i-1$. We claim that $\rank V=s_i-1$ is impossible. Indeed, $V^TV=D_{i-1}$ implies
\[
\rank(V^TV)=\rank D_{i-1}=s_{i-1}=s_i.
\]
But for any matrix $V$ over any field, $\rank(V^TV)\le \rank V$. Hence
\[
s_i=\rank(V^TV)\le \rank V.
\]
Therefore $\rank V$ cannot be $s_i-1$, since that would give $s_i\le s_i-1$, a contradiction. Consequently the only possible old rank is $\rank V=s_i$. For such an old root $V$, Lemma~\ref{lem:extension}(ii) states that when $d=0$, the number of rank-preserving extensions is
\[
2^{r-s_{i-1}},
\]
where $r=\rank V$. Substituting $r=s_i$ and $s_{i-1}=s_i$ gives
\[
2^{s_i-s_i}=2^0=1.
\]
Thus each old root of rank $s_i$ has exactly one extension that preserves the rank $s_i$. Hence
\[
h_i(0)=h_{i-1}(0).
\]

Now let $1\le t\le\lfloor z_i/2\rfloor$.  We must count how roots of rank $s_i+t$ arise
from old roots of rank either $s_i+t$ or $s_i+t-1$.

First consider an old root $V$ of rank
\[
r=s_i+t.
\]
Such a root contributes to $h_{i-1}(t)$. By Lemma $\ref{lem:extension}(ii)$,
when $d_i=0$, the number of extensions that preserve the rank $r$
is
\[
2^{r-s_{i-1}}=2^{(s_i+t)-s_i}=2^t.
\]
Each such extension yields a new root $U$ of rank $s_i+t$, hence
still with excess rank $t$. Thus the contribution from these roots is
\[
2^t\,h_{i-1}(t).
\]

Now consider an old root $V$ of rank
\[
r-1=s_i+t-1.
\]
Such a root contributes to $h_{i-1}(t-1)$. The total number of
extensions of such a $V$ is
\[
2^{(i-1)-(r-1)}
=2^{(i-1)-(s_i+t-1)}
=2^{i-s_i-t}
=2^{z_i-t}.
\]
Among these extensions, Lemma $\ref{lem:extension}(ii)$ tells us that
\[
2^{(r-1)-s_{i-1}}
=2^{(s_i+t-1)-s_i}
=2^{t-1}
\]
of them preserve the old rank $r-1$, and therefore do not contribute
to rank $s_i+t$. The remaining
\[
2^{z_i-t}-2^{t-1}
\]
extensions increase the rank by one, producing a new root of rank
$s_i+t$. Therefore the contribution from old roots of excess rank
$t-1$ is
\[
\bigl(2^{z_i-t}-2^{t-1}\bigr)h_{i-1}(t-1).
\]
Adding the two contributions gives
\[
h_i(t)
=2^t h_{i-1}(t)
+\bigl(2^{z_i-t}-2^{t-1}\bigr)h_{i-1}(t-1).
\]
For the stated range, $2t\le z_i$, and hence
$z_i-t\ge t$.  It follows that
\[
 2^{z_i-t}-2^{t-1}\ge 2^t-2^{t-1}=2^{t-1}>0,
\]
so every displayed coefficient is a nonnegative integer.

It remains to verify the support assertion. Suppose $t>\lfloor z_i/2\rfloor$, and consider the recurrence
\[
h_i(t)=2^t h_{i-1}(t)+\bigl(2^{z_i-t}-2^{t-1}\bigr)h_{i-1}(t-1).
\]
We show that both summands vanish.

First, the rank-preserving contribution vanishes. Since $z_i=z_{i-1}+1$, we have
\[
\left\lfloor \frac{z_i}{2}\right\rfloor
\ge
\left\lfloor \frac{z_i-1}{2}\right\rfloor
=
\left\lfloor \frac{z_{i-1}}{2}\right\rfloor .
\]
Thus $t>\lfloor z_i/2\rfloor$ implies $t>\lfloor z_{i-1}/2\rfloor$, and by the induction hypothesis $h_{i-1}(t)=0$. Hence $2^t h_{i-1}(t)=0$.

For the rank-increasing contribution, we distinguish two cases according to the parity of $z_i$.

If $z_i=2k$ is even, then $t>\lfloor z_i/2\rfloor=k$, so $t\ge k+1$. Consequently,
\[
t-1\ge k > k-1 = \left\lfloor \frac{2k-1}{2}\right\rfloor
= \left\lfloor \frac{z_{i-1}}{2}\right\rfloor .
\]
Therefore $h_{i-1}(t-1)=0$ by induction, and the second summand vanishes.

If $z_i=2k+1$ is odd, then $t>\lfloor z_i/2\rfloor=k$, so $t\ge k+1$. If $t=k+1$, then the coefficient of $h_{i-1}(t-1)$ is
\[
2^{z_i-t}-2^{t-1}
=
2^{(2k+1)-(k+1)}-2^{k}
=
2^k-2^k
=
0.
\]
Thus the second summand vanishes regardless of $h_{i-1}(t-1)$. If $t\ge k+2$, then
\[
t-1\ge k+1 > k = \left\lfloor \frac{2k}{2}\right\rfloor
= \left\lfloor \frac{z_{i-1}}{2}\right\rfloor ,
\]
so $h_{i-1}(t-1)=0$ by induction. Hence the second summand also vanishes.

In all cases both summands are zero, so $h_i(t)=0$ whenever $t>\lfloor z_i/2\rfloor$. 

Finally, every upper triangular root of $D_n$ is counted exactly once
by the sum of the admissible $h_n(t)$. Hence
\[
\nu(D_n)=\sum_{t=0}^{\lfloor z_n/2\rfloor}h_n(t).
\]
This completes the proof.
\end{proof}
\begin{remark}
The state range has at most $1+\lfloor i/2\rfloor$ entries at stage
$i$.  Consequently the complete rank distribution is obtained using
$O(n^2)$ additions and multiplications of integers.  This is an
arithmetic-operation count; since the integers themselves grow with
$n$, it is not a bit-complexity estimate.  The executable recurrence
is recorded in \Cref{app:diagonal-code}.
\end{remark}

\subsection[Relation to the rank-refined zero-fiber numbers]
{Relation to the Numbers $C_k(r)$}

\begin{corollary}
\label{cor:lpn}
If
\[
D_n=\diag(\underbrace{1,\ldots,1}_{s},
\underbrace{0,\ldots,0}_{n-s}),
\]
then
\[
h_n(t)=|C_{n-s}(t)|,
\qquad
\nu(D_n)=|C_{n-s}|.
\]
\end{corollary}

\begin{proof}
For a nonnegative integer $k$ and $t\ge0$, let $C_k(t)$ denote
the set of upper triangular matrices $U\in\T_k(\F_2)$ such that
\[
U^TU=0_k
\qquad\text{and}\qquad
\rank U=t.
\]
Thus $C_k=\bigcup_{t\ge0} C_k(t)$ and
\[
|C_k|=\sum_{t\ge0}|C_k(t)|.
\]

We prove the corollary by induction on the number of trailing zero
diagonal entries.

Starting from the empty matrix, we have $h_0(0)=1$. For $i=1,\dots,s$,
all diagonal entries added so far are $1$, hence $s_i=i$ and therefore
\[
z_i=i-s_i=0
\]
for every $1\le i\le s$.  The support assertion in
\Cref{thm:diagonal-recurrence} therefore leaves only the admissible
state $t=0$.  Since $d_i=1$, its recurrence is
\[
h_i(0)=2^{0}h_{i-1}(0)=h_{i-1}(0).
\]
By induction,
\[
h_s(0)=h_0(0)=1,
\qquad\text{and}\qquad
h_s(t)=0\quad\text{for }t\ne0.
\]

For $j\ge0$, let
\[
g_j(t)=h_{s+j}(t).
\]
We claim that $g_j(t)=|C_j(t)|$. This will prove the corollary by
taking $j=n-s$.

When $j=0$, we have already shown that
\[
g_0(0)=h_s(0)=1,
\qquad
g_0(t)=0\quad(t\ne0).
\]
On the other hand, the only $0\times0$ upper triangular matrix is the
empty matrix, which has rank $0$ and satisfies the zero-matrix
condition vacuously. Hence
\[
|C_0(0)|=1,
\qquad
|C_0(t)|=0\quad(t\ne0).
\]
Thus $g_0(t)=|C_0(t)|$ for all $t$.

Now suppose $j\ge1$. Since the $j$-th trailing zero is added at
stage $i=s+j$, we have $d_i=0$. Moreover,
\[
s_i=s_{i-1}=s,
\qquad
z_i=i-s_i=(s+j)-s=j.
\]
By \Cref{thm:diagonal-recurrence},
\[
 g_j(0)=g_{j-1}(0),
\]
and, for $1\le t\le\lfloor j/2\rfloor$,
\[
g_j(t)
=2^t g_{j-1}(t)
+\bigl(2^{j-t}-2^{t-1}\bigr)g_{j-1}(t-1),
\]
while $g_j(t)=0$ for $t>\lfloor j/2\rfloor$.

We must check that the numbers $|C_j(t)|$ satisfy exactly the same
recurrence. We derive it directly from the one-step extension.

Take $Y\in C_j(t)$ and write
\[
Y=\begin{pmatrix} Y_0 & w\\ 0 & c \end{pmatrix},
\qquad Y_0\in\T_{j-1}(\F_2),\quad w\in\F_2^{j-1},\quad c\in\F_2.
\]
The condition $Y^TY=0$ is equivalent to
\[
Y_0^TY_0=0,\qquad Y_0^Tw=0,\qquad c=w^Tw.
\]
Let $W=\Col Y_0$. Since $Y_0^TY_0=0$, the column space $W$ is
totally isotropic: $W\subseteq W^\perp$, where $W^\perp$ is the
orthogonal complement under the standard dot product. Thus
\[
\dim W^\perp=(j-1)-\rank Y_0.
\]

If $\rank Y_0=t$, then to preserve rank $t$ we need $w\in W$ and
$c=0$. For every $w\in W$, the isotropic condition gives
$w^Tw=0$, so $c=0$ is forced. Since $|W|=2^t$, there are exactly
$2^t$ rank-preserving extensions of $Y_0$.

For $t=0$, the preceding paragraph gives exactly one extension of the
unique rank-zero old root, and there is no old root of rank $-1$.
Consequently
\[
 |C_j(0)|=|C_{j-1}(0)|.
\]

Now let $t\ge1$.  If $\rank Y_0=t-1$, then
$\dim W^\perp=(j-1)-(t-1)=j-t$, so there
are $2^{j-t}$ possible choices of $w$ satisfying $Y_0^Tw=0$.
Among these, exactly $|W|=2^{t-1}$ choices have $w\in W$, which
preserve the rank $t-1$; the remaining $2^{j-t}-2^{t-1}$ choices
have $w\notin W$, and for each such $w$ the resulting matrix $Y$
has rank $t$. Hence
\[
|C_j(t)|
=2^t |C_{j-1}(t)|
+\bigl(2^{j-t}-2^{t-1}\bigr)|C_{j-1}(t-1)|,
\]

Finally, $Y^TY=0$ implies that $\Col Y$ is a totally isotropic
subspace of the nondegenerate $j$-dimensional dot-product space.
Because $\Col Y\subseteq(\Col Y)^\perp$, taking dimensions gives
$2\rank Y\le j$.  Hence $C_j(t)=\varnothing$ for
$t>\lfloor j/2\rfloor$.  The sequences $g_j(t)$ and $|C_j(t)|$
therefore have the same support as well as the same transition formulas.

Thus $g_j(t)$ and $|C_j(t)|$ satisfy the same recurrence and the
same initial conditions. By induction on $j$, we conclude that
\[
g_j(t)=|C_j(t)|
\]
for all $j\ge0$ and all $t\ge0$.

Taking $j=n-s$, we obtain
\[
h_n(t)=h_{s+(n-s)}(t)=g_{n-s}(t)=|C_{n-s}(t)|.
\]
Finally, summing over all $t$ gives
\[
\nu(D_n)
=\sum_{t\ge0}h_n(t)
=\sum_{t\ge0}|C_{n-s}(t)|
=|C_{n-s}|.
\]
This completes the proof.
\end{proof}
\begin{corollary}
\label{cor:not-nullity}
The Cholesky root count of a diagonal matrix is not determined solely
by its rank or nullity.
\end{corollary}

\begin{proof}
We exhibit two diagonal matrices with the same rank and the same
nullity but different numbers of upper triangular Cholesky roots.

Consider the two $2\times 2$ diagonal matrices over $\F_2$
\[
D=\begin{pmatrix}1&0\\0&0\end{pmatrix},
\qquad
E=\begin{pmatrix}0&0\\0&1\end{pmatrix}.
\]
Both have exactly one nonzero diagonal entry, so
\[
\rank D=\rank E=1.
\]
Since both matrices are $2\times 2$, the rank--nullity theorem gives
\[
\dim\Null D=2-\rank D=1,
\qquad
\dim\Null E=2-\rank E=1.
\]
Thus $D$ and $E$ agree in both rank and nullity.

Now compute their Cholesky root counts explicitly. A general upper
triangular $2\times 2$ matrix over $\F_2$ has the form
\[
U=\begin{pmatrix}u&x\\0&c\end{pmatrix},
\qquad u,x,c\in\F_2.
\]
Over $\F_2$ we have $u^2=u$ and $c^2=c$, so
\[
U^TU
=
\begin{pmatrix}
u^2 & ux\\
ux & x^2+c^2
\end{pmatrix}
=
\begin{pmatrix}
u & ux\\
ux & x+c
\end{pmatrix}.
\]

For the roots of $D$, we need $U^TU=D$, that is,
\[
\begin{pmatrix}
u & ux\\
ux & x+c
\end{pmatrix}
=
\begin{pmatrix}
1&0\\
0&0
\end{pmatrix}.
\]
Comparing the $(1,1)$-entries gives $u=1$. Comparing the
$(1,2)$-entries gives $ux=0$, hence $x=0$ because $u=1$.
Finally, comparing the $(2,2)$-entries gives $x+c=0$, and with
$x=0$ this yields $c=0$. Therefore the only solution is
\[
U=\begin{pmatrix}1&0\\0&0\end{pmatrix}.
\]
Hence
\[
\nu(D)=1.
\]

For the roots of $E$, we need $U^TU=E$, that is,
\[
\begin{pmatrix}
u & ux\\
ux & x+c
\end{pmatrix}
=
\begin{pmatrix}
0&0\\
0&1
\end{pmatrix}.
\]
Comparing the $(1,1)$-entries gives $u=0$. The $(1,2)$-entry
condition $ux=0$ is then automatically satisfied. Comparing the
$(2,2)$-entries gives
\[
x+c=1.
\]
Over $\F_2$, this equation has exactly two solutions:
\[
x=0,\ c=1
\qquad\text{and}\qquad
x=1,\ c=0.
\]
The corresponding matrices are
\[
\begin{pmatrix}0&0\\0&1\end{pmatrix}
\qquad\text{and}\qquad
\begin{pmatrix}0&1\\0&0\end{pmatrix}.
\]
Therefore
\[
\nu(E)=2.
\]

Since $D$ and $E$ have the same rank and the same nullity but
different Cholesky root counts, we conclude that the Cholesky root
count of a diagonal matrix cannot be determined solely by its rank or
nullity.
\end{proof}
\begin{remark}
Thus the natural answer to the question of whether the count can be
expressed in terms of the numbers $|C_k|$ is as follows.
\begin{itemize}
\item In the leading-ones--trailing-zeros case, the answer is exactly
      the single total $|C_{n-r}|$.
\item For a general diagonal ordering, a formula depending only on the
      total or only on nullity does not hold.
\item The complete answer requires the rank-refined transition structure
      obeyed by the numbers $|C_k(t)|$. A zero diagonal entry produces
      the same transition as in $C_k(t)$, while a one diagonal entry
      multiplies each excess-rank layer by $2^{z_i-t}$.
\end{itemize}
\end{remark}

\subsection{Low-Order Examples}

\begin{example}
For $D=\diag(0,1)$, the first step adds a zero and gives
$h_1(0)=1$. The second step adds a one, and since $z_2=1$,
\[
h_2(0)=2^{1}h_1(0)=2.
\]
Therefore
\[
\nu(\diag(0,1))=2.
\]
The two roots are
\[
\begin{pmatrix}0&0\\0&1\end{pmatrix},
\qquad
\begin{pmatrix}0&1\\0&0\end{pmatrix}.
\]
\end{example}

\begin{example}
For $4\times4$ diagonal matrices, represent the diagonal entries by the
binary string $d_1d_2d_3d_4$. The recursion gives the following table,
where a star means that the last entry may be either $0$ or $1$.
\[
\begin{array}{c|cccccccc}
\toprule
(d_1d_2d_3d_4)
&000*&001*&010*&011*&100*&101*&110*&111*\\
\midrule
\nu(D_4)
&28&20&12&8&6&4&2&1\\
\bottomrule
\end{array}
\]
In particular, the last diagonal entry does not change the total number
of roots, but it does change the rank distribution of the roots. This
is because, in the last step, the total number of extensions of each
old root is independent of $d_4$, while $d_4$ determines which
extensions preserve the rank.
\end{example}

The corresponding state-transition implementation is given in
\Cref{app:diagonal-code}.

\section[An Explicit Rank-Preserving Bijection]
{An Explicit Rank-Preserving Bijection
Between $B_n(r)$ and $C_n(r)$}
\label{sec:bijection}

\subsection{The Two Nested Spaces Behind the Same Recursion}

Take
\[
X=
\begin{pmatrix}
X_0&v\\
0&0
\end{pmatrix}\in B_n.
\]
The condition $X^2=0$ is equivalent to
\[
X_0^2=0,\qquad v\in\Null X_0,
\]
and also
\[
\Col X_0\subseteq\Null X_0.
\]

On the other hand, take
\[
Y=
\begin{pmatrix}
Y_0&w\\
0&c
\end{pmatrix}\in C_n.
\]
The condition $Y^TY=0$ is equivalent to
\[
Y_0^TY_0=0,\qquad
w\in\Null Y_0^T,\qquad
c=w^Tw,
\]
while
\[
\Col Y_0\subseteq\Null Y_0^T.
\]

If $\rank X_0=\rank Y_0=r$, then the two pairs of nested spaces have
exactly the same dimensions:
\[
\begin{array}{c|cc}
&\text{inner dimension}&\text{outer dimension}\\
\hline
\Col X_0\subseteq\Null X_0&r&(n-1)-r\\
\Col Y_0\subseteq\Null Y_0^T&r&(n-1)-r.
\end{array}
\]

\subsection{Canonical Adapted Bases}

\begin{definition}
\label{def:adapted-basis}
For a nested pair of subspaces
\[
S\subseteq T\subseteq\F_2^m,
\qquad \dim S=a,\quad \dim T=b,
\]
define its canonical adapted basis as follows.
\begin{enumerate}[label=\textup{(\arabic*)}]
\item First take the unique reduced row echelon basis
      $(s_1,\ldots,s_a)$ of $S$.
\item Then scan the vectors of $T$ in a fixed lexicographic order;
      whenever a vector is not already in the span of the previously
      chosen vectors, append it, until
      \[
      (s_1,\ldots,s_a,t_{a+1},\ldots,t_b)
      \]
      is a basis of $T$.
\end{enumerate}
\end{definition}

This choice is determined entirely by the standard coordinate order and
therefore yields an executable canonical algorithm.  The following lemma
does not assert uniqueness among all isomorphisms carrying $S$ onto
$S'$; rather, it singles out one isomorphism uniquely by the prescribed
adapted bases.

\begin{lemma}
\label{lem:nested-iso}
Suppose
\[
S\subseteq T,\qquad S'\subseteq T',
\]
and
\[
\dim S=\dim S',\qquad \dim T=\dim T'.
\]
Then mapping corresponding vectors of the two canonical adapted bases
to each other determines a linear isomorphism
\[
\psi:T\longrightarrow T'
\]
such that $\psi(S)=S'$.  It is the unique linear map sending every
vector of the canonical adapted basis of $T$ to the corresponding
vector of the canonical adapted basis of $T'$.
\end{lemma}
\begin{proof}
Let
\[
\mathcal B=(s_1,\ldots,s_a,t_{a+1},\ldots,t_b)
\]
be the canonical adapted basis of $T$ associated with the nested pair
$S\subseteq T$, as constructed in Definition~\ref{def:adapted-basis}.
Here
\[
a=\dim S,\qquad b=\dim T,
\]
and the first $a$ vectors form the unique reduced row echelon basis
of $S$. Similarly, let
\[
\mathcal B'=(s'_1,\ldots,s'_a,t'_{a+1},\ldots,t'_b)
\]
be the canonical adapted basis of $T'$ associated with
$S'\subseteq T'$.

Because both pairs have the same dimensions, the two bases $\mathcal B$
and $\mathcal B'$ have exactly the same number of vectors. We define a
linear map $\psi:T\to T'$ by specifying its values on the basis
$\mathcal B$:
\[
\psi(s_i)=s'_i\quad (1\le i\le a),
\qquad
\psi(t_j)=t'_j\quad (a+1\le j\le b).
\]
By the universal property of bases, this prescription extends uniquely
to a well-defined linear map on all of $T$. Thus $\psi$ is uniquely
determined by the canonical adapted bases.

We now verify that $\psi$ is an isomorphism. Since $\mathcal B'$
is a basis of $T'$, the set
\[
\{\psi(s_1),\ldots,\psi(s_a),\psi(t_{a+1}),\ldots,\psi(t_b)\}
=
\mathcal B'
\]
is linearly independent and spans $T'$. Therefore $\psi$ maps a
basis of $T$ bijectively onto a basis of $T'$, which implies that
$\psi$ is both injective and surjective. Hence $\psi$ is a linear
isomorphism.

It remains to show that $\psi(S)=S'$. Since
$s_1,\ldots,s_a$ is a basis of $S$, every vector $x\in S$ can be
written uniquely as
\[
x=\sum_{i=1}^a \alpha_i s_i.
\]
Applying $\psi$ gives
\[
\psi(x)=\sum_{i=1}^a \alpha_i \psi(s_i)
=\sum_{i=1}^a \alpha_i s'_i.
\]
Because $s'_1,\ldots,s'_a$ is a basis of $S'$, it follows that
$\psi(x)\in S'$. Hence $\psi(S)\subseteq S'$. Conversely, since
$\psi$ is an isomorphism and both $S$ and $S'$ have dimension
$a$, the subspace $\psi(S)$ is an $a$-dimensional subspace of
$S'$. As $S'$ itself is $a$-dimensional, we must have
$\psi(S)=S'$.

Thus the canonical-basis prescription determines exactly one linear
isomorphism, and that isomorphism maps $S$ onto $S'$.
\end{proof}
\subsection{Recursive Construction of the Bijection}

\begin{theorem}[Explicit rank-preserving bijection]
\label{thm:explicit-bijection}
For every $n\ge0$ and every $0\le r\le n$, there is an explicit
invertible map
\[
\Phi_n:B_n(r)\longrightarrow C_n(r).
\]
\end{theorem}

\begin{proof}
We prove the theorem by induction on $n$.

For $n=0$, both $B_0(0)$ and $C_0(0)$ contain the empty matrix, while
$B_0(r)=C_0(r)=\varnothing$ for $r>0$; the assertion follows.  For
$n=1$, an upper triangular $1\times1$ matrix is a scalar.  The
condition $X^2=0$ forces $X=0$, and the condition $U^TU=0$ likewise
forces $U=0$.  Hence both $B_1(0)$ and $C_1(0)$ contain only the zero
matrix and both rank-$r$ sets are empty for $r>0$.

Assume that for every rank $0\le s\le n-1$ we have already constructed an explicit bijection
\[
\Phi_{n-1}:B_{n-1}(s)\longrightarrow C_{n-1}(s)
\]
that preserves rank. We will use the same symbol $\Phi_{n-1}$ to denote the induced bijection on the union $B_{n-1}\to C_{n-1}$.

Take an arbitrary $X\in B_n(r)$. Since $X$ is upper triangular and $X^2=0$, the bottom-right entry of $X$ must be $0$. Indeed, writing $X=(x_{ij})$, the bottom-right entry of $X^2$ is
\[
(X^2)_{nn}=x_{nn}^2,
\]
because all other terms in the last row or last column are zero. The condition $X^2=0$ gives $x_{nn}^2=0$, and since we are over $\F_2$, this forces $x_{nn}=0$. Therefore $X$ can be written uniquely in the block form
\[
X=
\begin{pmatrix}
X_0 & v\\
0 & 0
\end{pmatrix},
\]
where $X_0\in\T_{n-1}(\F_2)$ and $v\in\F_2^{n-1}$.

Because $X$ is upper triangular and $X^2=0$, expanding the block product gives
\[
X^2=
\begin{pmatrix}
X_0^2 & X_0v\\
0 & 0
\end{pmatrix}=0,
\]
so
\[
X_0^2=0
\qquad\text{and}\qquad
v\in\Null X_0.
\]
In particular, $X_0\in B_{n-1}$. Let
\[
Y_0=\Phi_{n-1}(X_0).
\]
By the induction hypothesis,
\[
\rank Y_0=\rank X_0
\]
and $Y_0\in C_{n-1}$, that is,
\[
Y_0^TY_0=0.
\]

Now consider the two pairs of nested subspaces
\[
\Col X_0\subseteq\Null X_0
\qquad\text{and}\qquad
\Col Y_0\subseteq\Null Y_0^T.
\]
The first inclusion follows from $X_0^2=0$, and the second from $Y_0^TY_0=0$. Since $\rank X_0=\rank Y_0$, both inner spaces have the same dimension, and both outer spaces have dimension $n-1-\rank X_0$. Thus Lemma $\ref{lem:nested-iso}$ applies, yielding a unique canonical linear isomorphism
\[
\psi_{X_0,Y_0}:\Null X_0\longrightarrow\Null Y_0^T
\]
such that
\[
\psi_{X_0,Y_0}(\Col X_0)=\Col Y_0.
\]

Define
\[
w=\psi_{X_0,Y_0}(v),\qquad c=w^Tw,
\]
and set
\[
\Phi_n(X)=
\begin{pmatrix}
Y_0 & w\\
0 & c
\end{pmatrix}.
\]

Since $v\in\Null X_0$, we have $w=\psi_{X_0,Y_0}(v)\in\Null Y_0^T$, so
\[
Y_0^Tw=0.
\]
Over $\mathbb F_2$, every element equals its own square, so $c^2=c$. Therefore
\[
w^Tw+c^2=w^Tw+c=w^Tw+w^Tw=0.
\]
Now compute
\[
\Phi_n(X)^T\Phi_n(X)
=
\begin{pmatrix}
Y_0^T & 0\\
w^T & c
\end{pmatrix}
\begin{pmatrix}
Y_0 & w\\
0 & c
\end{pmatrix}
=
\begin{pmatrix}
Y_0^TY_0 & Y_0^Tw\\
w^TY_0 & w^Tw+c^2
\end{pmatrix}.
\]
Using $Y_0^TY_0=0$, $Y_0^Tw=0$, and $w^Tw+c^2=0$, we obtain
\[
\Phi_n(X)^T\Phi_n(X)=0.
\]
Thus $\Phi_n(X)\in C_n$.

For the original matrix $X$, since its last column is
\[
\begin{pmatrix} v\\0 \end{pmatrix},
\]
we have
\[
\rank X=
\begin{cases}
\rank X_0,&v\in\Col X_0,\\
\rank X_0+1,&v\notin\Col X_0.
\end{cases}
\]
Indeed, the rank increases by one exactly when the last column is not in the column space of the previous columns.

For $Y=\Phi_n(X)$, the last column is
\[
\begin{pmatrix} w\\c \end{pmatrix}.
\]
If $w\in\Col Y_0$, then because $Y_0^TY_0=0$, the column space $\Col Y_0$ is totally isotropic, so $w^Tw=0$. Hence $c=w^Tw=0$, and the last column is
\[
\begin{pmatrix} w\\0 \end{pmatrix}\in\Col\begin{pmatrix}Y_0\\0\end{pmatrix}.
\]
Thus the rank of $Y$ equals $\rank Y_0$.

If $w\notin\Col Y_0$, then the last column of $Y$ is not in the span of its previous columns, regardless of the value of $c$. Hence
\[
\rank Y=\rank Y_0+1.
\]

Since $\psi_{X_0,Y_0}$ maps $\Col X_0$ exactly onto $\Col Y_0$, we have
\[
v\in\Col X_0\iff w\in\Col Y_0.
\]
Combining this with the induction hypothesis $\rank Y_0=\rank X_0$, we conclude
\[
\rank\Phi_n(X)=\rank X.
\]
Therefore $\Phi_n$ restricts to a map $B_n(r)\to C_n(r)$ for every $r$.

We now define the inverse map. Given
\[
Y=
\begin{pmatrix}
Y_0 & w\\
0 & c
\end{pmatrix}\in C_n(r),
\]
the condition $Y^TY=0$ implies
\[
Y_0^TY_0=0,\qquad Y_0^Tw=0,\qquad c=w^Tw.
\]
Thus $Y_0\in C_{n-1}$. By the induction hypothesis, there is a unique
\[
X_0=\Phi_{n-1}^{-1}(Y_0)\in B_{n-1}
\]
with $\rank X_0=\rank Y_0$.

The two nested-space pairs
\[
\Col X_0\subseteq\Null X_0
\qquad\text{and}\qquad
\Col Y_0\subseteq\Null Y_0^T
\]
again have the same inner and outer dimensions, so the same canonical isomorphism
\[
\psi_{X_0,Y_0}:\Null X_0\longrightarrow\Null Y_0^T
\]
is available. Since $w\in\Null Y_0^T$, its preimage
\[
v=\psi_{X_0,Y_0}^{-1}(w)
\]
lies in $\Null X_0$. Define
\[
\Psi_n(Y)=
\begin{pmatrix}
X_0 & v\\
0 & 0
\end{pmatrix}.
\]
Then
\[
\Psi_n(Y)^2=
\begin{pmatrix}
X_0^2 & X_0v\\
0 & 0
\end{pmatrix}=0
\]
because $X_0^2=0$ and $v\in\Null X_0$, so $\Psi_n(Y)\in B_n$.

Moreover, by construction,
\[
\Psi_n(\Phi_n(X))=X
\]
for every $X\in B_n$, and
\[
\Phi_n(\Psi_n(Y))=Y
\]
for every $Y\in C_n$. Hence $\Phi_n$ is a bijection from $B_n$ to $C_n$, and since it preserves rank, it restricts to a bijection $B_n(r)\to C_n(r)$ for every $r$.

This completes the induction and the proof.
\end{proof}
\begin{remark}
The above bijection is canonical and computable with respect to the
standard coordinate order and the upper triangular structure. It does
not require choosing an arbitrary isomorphism for each input; rather,
it is uniquely determined by the RREF--lexicographic rule. However, if
``natural'' is required to mean functorial compatibility with every
linear isomorphism and complete independence of a coordinate ordering,
then naturality in this stronger sense remains an interesting question.
Since upper triangular matrices already fix a standard complete flag,
the present construction may be regarded as a natural flag-compatible
bijection in the combinatorial sense relevant to the problem.
\end{remark}

\section{Rank-Refined Zero Fibers over Arbitrary Finite Fields}
\label{sec:zero-fibers}

\subsection{Unified Notation and Block Recursion}

For an arbitrary finite field $\F_q$, define
\[
\begin{aligned}
B_{n,q}(r)
&=\{X\in\T_n(\F_q):X^2=0,\ \rank X=r\},\\
C_{n,q}(r)
&=\{U\in\T_n(\F_q):U^TU=0,\ \rank U=r\}.
\end{aligned}
\]
Write
\[
b_{n,r}^{(q)}=|B_{n,q}(r)|,\qquad
c_{n,r}^{(q)}=|C_{n,q}(r)|.
\]
We use the convention
\[
b_{0,0}^{(q)}=c_{0,0}^{(q)}=1,
\]
and all out-of-range values are zero.

\begin{proposition}
\label{prop:Bq}
For every finite field $\F_q$, one has
\[
 b_{n,0}^{(q)}=1,
 \qquad
 b_{n,r}^{(q)}=0\quad\text{if }r>\left\lfloor\frac n2\right\rfloor.
\]
For $1\le r\le\lfloor n/2\rfloor$,
\[
\boxed{
b_{n,r}^{(q)}
=q^r b_{n-1,r}^{(q)}
+\bigl(q^{n-r}-q^{r-1}\bigr)b_{n-1,r-1}^{(q)}.
}
\]
\end{proposition}

\begin{proof}
The only matrix of rank zero is the zero matrix, and it satisfies
$X^2=0$.  Hence $b_{n,0}^{(q)}=1$.  If $X^2=0$, then
\[
 \Col X\subseteq\Null X.
\]
Consequently, if $\rank X=r$, then
\[
 r=\dim\Col X\le\dim\Null X=n-r,
\]
so $2r\le n$.  This proves the asserted support bound.  We now fix
$1\le r\le\lfloor n/2\rfloor$ and argue by
decomposing an arbitrary matrix $X\in B_{n,q}(r)$ according to its
last row and column.

Let
\[
X=
\begin{pmatrix}
X_0 & v\\
0 & a
\end{pmatrix},
\]
where $X_0\in\T_{n-1}(\F_q)$, $v\in\F_q^{n-1}$, and $a\in\F_q$. Since $X$ is upper triangular, this block form is completely general. Direct multiplication gives
\[
X^2=
\begin{pmatrix}
X_0^2 & X_0v+av\\
0 & a^2
\end{pmatrix}
=
\begin{pmatrix}
X_0^2 & (X_0+aI_{n-1})v\\
0 & a^2
\end{pmatrix}.
\]
The condition $X^2=0$ is therefore equivalent to the three equations
\[
X_0^2=0,\qquad a^2=0,\qquad (X_0+aI_{n-1})v=0.
\]
Since $\F_q$ is a field, $a^2=0$ implies $a=0$. Consequently the last equation simplifies to
\[
X_0v=0,
\]
and the first equation becomes $X_0^2=0$. Thus
\[
X\in B_{n,q}(r)
\]
if and only if
\[
X_0\in B_{n-1,q}(s)
\quad\text{for some }s,
\qquad
v\in\Null X_0,
\qquad
a=0.
\]

Now we analyze how the rank of $X$ depends on $X_0$ and $v$. Since $a=0$, the last column of $X$ is
\[
\begin{pmatrix} v\\0 \end{pmatrix}.
\]
Hence
\[
\rank X =
\begin{cases}
\rank X_0, & v\in\Col X_0,\\[1mm]
\rank X_0+1, & v\notin\Col X_0.
\end{cases}
\]
Therefore, to obtain $\rank X=r$, the rank of $X_0$ must be either $r$ or $r-1$.

\paragraph{Case 1: $\rank X_0=r$.}
Then to keep the rank equal to $r$, we must have $v\in\Col X_0$. Since $\dim\Col X_0=r$, there are exactly $q^r$ choices of $v$ in $\Col X_0$, and each such $v$ automatically lies in $\Null X_0$ because $X_0^2=0$ implies $\Col X_0\subseteq\Null X_0$. Thus each matrix $X_0$ of rank $r$ contributes exactly $q^r$ matrices to $B_{n,q}(r)$. The number of such $X_0$ is $b_{n-1,r}^{(q)}$, giving the first term
\[
q^r b_{n-1,r}^{(q)}.
\]

\paragraph{Case 2: $\rank X_0=r-1$.}
Then we must have $v\notin\Col X_0$ to increase the rank to $r$. Additionally, $v$ must satisfy $X_0v=0$, i.e.\ $v\in\Null X_0$. By the rank--nullity theorem,
\[
\dim\Null X_0=(n-1)-\rank X_0=(n-1)-(r-1)=n-r,
\]
so
\[
|\Null X_0|=q^{n-r}.
\]
Among these null vectors, exactly those lying in $\Col X_0$ would preserve the rank $r-1$; since $\dim\Col X_0=r-1$, there are $q^{r-1}$ such vectors. Hence the number of choices of $v$ that increase the rank to $r$ is
\[
q^{n-r}-q^{r-1}.
\]
The number of such $X_0$ is $b_{n-1,r-1}^{(q)}$, giving the second term
\[
\bigl(q^{n-r}-q^{r-1}\bigr)b_{n-1,r-1}^{(q)}.
\]

No other rank of $X_0$ can produce $\rank X=r$. Therefore, summing over the two possible ranks of $X_0$, we obtain
\[
b_{n,r}^{(q)}
=
q^r b_{n-1,r}^{(q)}
+
\bigl(q^{n-r}-q^{r-1}\bigr)b_{n-1,r-1}^{(q)}.
\]

Finally, the convention
\[
b_{0,0}^{(q)}=1,\qquad b_{0,r}^{(q)}=0\quad(r\ne0)
\]
provides the base case for the recursion, and all out-of-range indices are treated as zero. This completes the proof.
\end{proof}
\subsection{All Even-Characteristic Finite Fields}

\begin{theorem}
\label{thm:even-recurrence}
If $q=2^e$, then
\[
 c_{n,0}^{(q)}=1,
 \qquad
 c_{n,r}^{(q)}=0\quad\text{if }r>\left\lfloor\frac n2\right\rfloor.
\]
For $1\le r\le\lfloor n/2\rfloor$,
\[
\boxed{
c_{n,r}^{(q)}
=q^r c_{n-1,r}^{(q)}
+\bigl(q^{n-r}-q^{r-1}\bigr)c_{n-1,r-1}^{(q)}.
}
\]
Consequently,
\[
\boxed{
b_{n,r}^{(q)}=c_{n,r}^{(q)}
}
\]
for all $n$ and $r$.
\end{theorem}

\begin{proof}
The only rank-zero matrix is the zero matrix, which satisfies
$U^TU=0$.  Thus $c_{n,0}^{(q)}=1$.  If $U^TU=0$ and
$W=\Col U$, then for arbitrary $Ux,Uy\in W$,
\[
 (Ux)^T(Uy)=x^TU^TUy=0.
\]
Thus $W\subseteq W^\perp$.  Since the standard dot product is
nondegenerate, $\dim W^\perp=n-\dim W$, and hence
\[
 \rank U\le n-\rank U.
\]
This proves that $c_{n,r}^{(q)}=0$ when $r>\lfloor n/2\rfloor$.
Fix $1\le r\le\lfloor n/2\rfloor$.  We analyze the
possible one-step extensions of an upper triangular matrix
$U$ satisfying $U^TU=0$. Write $U$ in block form
\[
U=
\begin{pmatrix}
U_0 & w\\
0 & c
\end{pmatrix},
\qquad
U_0\in\T_{n-1}(\F_q),\quad w\in\F_q^{n-1},\quad c\in\F_q.
\]
Then
\[
U^TU
=
\begin{pmatrix}
U_0^T & 0\\
w^T & c
\end{pmatrix}
\begin{pmatrix}
U_0 & w\\
0 & c
\end{pmatrix}
=
\begin{pmatrix}
U_0^TU_0 & U_0^Tw\\
w^TU_0 & w^Tw+c^2
\end{pmatrix}.
\]
Thus $U^TU=0$ is equivalent to the three conditions
\[
U_0^TU_0=0,\qquad U_0^Tw=0,\qquad w^Tw+c^2=0.
\]
The first condition says $U_0\in C_{n-1,q}$. The second says
$w\in\Null U_0^T$. If we let
\[
W=\Col U_0,
\]
then over $\F_q$ with the standard dot product,
\[
\Null U_0^T=W^\perp.
\]
Moreover, the condition $U_0^TU_0=0$ implies that $W$ is totally
isotropic, that is,
\[
W\subseteq W^\perp,
\]
because for arbitrary $y=U_0x$ and $z=U_0x'$ in $W$,
\[
y^Tz=(U_0x)^T(U_0x')=x^TU_0^TU_0x'=0.
\]
In particular, every vector in $W$ has squared length zero.

Now assume $q=2^e$, so that $\F_q$ has characteristic $2$. Then the map
\[
\varphi:\F_q\to\F_q,\qquad x\mapsto x^2
\]
is bijective. Hence for every $w\in W^\perp$, the equation
\[
c^2=-w^Tw
\]
has a unique solution
\[
c=\bigl(-w^Tw\bigr)^{q/2}.
\]
Therefore, once $U_0$ and $w$ are chosen with $w\in W^\perp$, the entry
$c$ is uniquely determined. Thus the number of extensions of a fixed
$U_0$ is exactly $|W^\perp|$.

We now determine how the rank of $U$ depends on $U_0$ and $w$. The
matrix
\[
U=
\begin{pmatrix}
U_0 & w\\
0 & c
\end{pmatrix}
\]
has its first $n-1$ columns equal to the columns of
\[
\begin{pmatrix}
U_0\\0
\end{pmatrix},
\]
whose first $n-1$ coordinates span $W$ and whose last coordinate is
zero. Therefore
\[
\rank U=
\begin{cases}
\rank U_0, & w\in W \text{ and } c=0,\\[1mm]
\rank U_0+1, & \text{otherwise}.
\end{cases}
\]
If $w\in W$, then since $W$ is totally isotropic, $w^Tw=0$. The
equation $c^2=-w^Tw$ gives $c^2=0$, hence $c=0$ because $\F_q$ is a
field. Thus for $w\in W$ we automatically have $c=0$, and the rank is
preserved.

Now suppose $\rank U_0=r$. Then $\dim W=r$, so $|W|=q^r$. For each
$w\in W$, the resulting $U$ has rank $r$. These are exactly the
rank-preserving extensions, and their number is $q^r$. Therefore the
contribution from old roots of rank $r$ is
\[
q^r c_{n-1,r}^{(q)}.
\]

Next suppose $\rank U_0=r-1$. Then
\[
\dim W=r-1,\qquad
\dim W^\perp=(n-1)-(r-1)=n-r.
\]
Thus there are $q^{n-r}$ possible choices of $w\in W^\perp$. Among
these, exactly the $q^{r-1}$ vectors in $W$ preserve the old rank
$r-1$. The remaining
\[
q^{n-r}-q^{r-1}
\]
vectors $w\notin W$ make the last column linearly independent of the
previous columns, so the rank increases to $r$. Therefore the
contribution from old roots of rank $r-1$ is
\[
\bigl(q^{n-r}-q^{r-1}\bigr)c_{n-1,r-1}^{(q)}.
\]

No other rank of $U_0$ can produce rank $r$ after a one-step extension.
Hence
\[
c_{n,r}^{(q)}
=
q^r c_{n-1,r}^{(q)}
+
\bigl(q^{n-r}-q^{r-1}\bigr)c_{n-1,r-1}^{(q)}.
\]

Finally, the initial conditions for $c^{(q)}$ are the same as those for
$b^{(q)}$:
\[
c_{0,0}^{(q)}=b_{0,0}^{(q)}=1,
\qquad
c_{0,r}^{(q)}=b_{0,r}^{(q)}=0\quad(r\ne0).
\]
Since the rank-zero values agree and the recurrences agree for every
$r\ge1$, induction on $n$ gives
\[
b_{n,r}^{(q)}=c_{n,r}^{(q)}
\]
for all $n\ge0$ and all $r$.
\end{proof}
The following corollary records the total equinumerosity of the three families of upper triangular matrices in even characteristic.
\begin{corollary}
\label{cor:three-even}
For every $q=2^e$, the three families
\begin{itemize}
\item upper triangular matrices with square zero,
\item upper triangular Cholesky roots of the zero matrix,
\item upper triangular matrices with square equal to the identity
\end{itemize}
are equinumerous in total.
\end{corollary}

\begin{proof}
Let $q=2^e$. Recall the definitions
\[
\begin{aligned}
A_{n,q} &= \{X\in\T_n(\F_q): X^2=I_n\},\\
B_{n,q} &= \{X\in\T_n(\F_q): X^2=0\},\\
C_{n,q} &= \{U\in\T_n(\F_q): U^TU=0\}.
\end{aligned}
\]
By \Cref{thm:even-recurrence}, for every $n$ and every rank $r$, we have
\[
|B_{n,q}(r)|=|C_{n,q}(r)|.
\]
Summing over all $r$ gives
\[
|B_{n,q}|=|C_{n,q}|.
\]

It remains to show that $|A_{n,q}|=|B_{n,q}|$. Since $q=2^e$, the field $\F_q$ has characteristic $2$. For any matrix $X\in\T_n(\F_q)$, define
\[
\Phi(X)=X+I_n.
\]
Because $I_n$ is upper triangular and the sum of upper triangular matrices is upper triangular, $\Phi(X)\in\T_n(\F_q)$.

We claim that $\Phi$ maps $B_{n,q}$ bijectively onto $A_{n,q}$. Suppose first that $X\in B_{n,q}$, so $X^2=0$. Then
\[
\Phi(X)^2=(X+I_n)^2.
\]
Expanding in characteristic $2$, we have
\[
(X+I_n)^2=X^2+2X+I_n^2.
\]
Since $2=0$ in $\F_q$ and $I_n^2=I_n$, this simplifies to
\[
(X+I_n)^2=X^2+I_n=0+I_n=I_n.
\]
Thus $\Phi(X)\in A_{n,q}$, so $\Phi$ maps $B_{n,q}$ into $A_{n,q}$.

Conversely, if $Y\in A_{n,q}$, then $Y^2=I_n$. Define
\[
\Psi(Y)=Y+I_n.
\]
Then
\[
\Psi(Y)^2=(Y+I_n)^2
=Y^2+2Y+I_n^2
=I_n+0+I_n
=2I_n
=0,
\]
because $2=0$ in characteristic $2$. Hence $\Psi(Y)\in B_{n,q}$.

Now observe that $\Phi$ and $\Psi$ are inverse maps. Indeed, for any $X\in\T_n(\F_q)$,
\[
\Psi(\Phi(X))=\Psi(X+I_n)=(X+I_n)+I_n=X+2I_n=X,
\]
and similarly
\[
\Phi(\Psi(Y))=\Phi(Y+I_n)=(Y+I_n)+I_n=Y+2I_n=Y.
\]
Therefore $\Phi$ is a bijection between $B_{n,q}$ and $A_{n,q}$, and consequently
\[
|A_{n,q}|=|B_{n,q}|.
\]

Combining the two equalities yields
\[
|A_{n,q}|=|B_{n,q}|=|C_{n,q}|,
\]
which proves that the three families are equinumerous in total.
\end{proof}
\begin{remark}
The explicit construction of \Cref{thm:explicit-bijection} also extends
to $\F_{2^e}$: one only needs to define the last diagonal entry as
\[
c=(-w^Tw)^{q/2}.
\]
The nested-space isomorphism can be constructed canonically after
fixing an $\F_2$-basis and a fixed lexicographic order.
\end{remark}

\subsection{Number of Extension Pairs in Odd Characteristic}
The quadratic and orthogonal structures used below belong to the
classical finite-field theory represented, for example, by MacWilliams
\cite{MacWilliams1969}.  For the rest of this section assume that $q$ is
odd, and let $\chi:\F_q\to\{-1,0,1\}$ be the quadratic character.  The
following lemma is a standard consequence of
\cite[Theorem~6.26]{LidlNiederreiter} and
\cite[Theorem~6.27]{LidlNiederreiter}; for completeness and the reader's
convenience, we include a proof here.

\begin{lemma}
\label{lem:quadratic-zeros}
Let
\[
Q(x)=x^TAx
\]
be a nondegenerate quadratic form on $\F_q^d$. Put
\[
Z(Q)=|\{x\in\F_q^d:Q(x)=0\}|.
\]
Then
\[
Z(Q)=
\begin{cases}
q^{d-1},&d\text{ odd},\\[1mm]
q^{d-1}
+(q-1)\chi\!\left((-1)^{d/2}\det A\right)q^{(d-2)/2},
&d\text{ even}.
\end{cases}
\]
\end{lemma}

\begin{proof}
Let $\psi$ be a nontrivial additive character of $\F_q$. By the
orthogonality relation for additive characters,
\[
\frac1q\sum_{t\in\F_q}\psi(t\alpha)=
\begin{cases}
1,&\alpha=0,\\
0,&\alpha\ne0.
\end{cases}
\]
Applying this to $\alpha=Q(x)$, we obtain
\[
Z(Q)
=\sum_{x\in\F_q^d}\ind_{Q(x)=0}
=\sum_{x\in\F_q^d}\frac1q\sum_{t\in\F_q}\psi(tQ(x))
=\frac1q\sum_{t\in\F_q}\sum_{x\in\F_q^d}\psi(tQ(x)).
\]
The term with $t=0$ contributes
\[
\frac1q\sum_{x\in\F_q^d}\psi(0)=\frac1q\cdot q^d=q^{d-1}.
\]

Now assume $t\ne0$. Since $q$ is odd, every nondegenerate quadratic
form over $\F_q$ is equivalent by an invertible linear change of
variables to a diagonal form
\[
Q(x)=a_1x_1^2+\cdots+a_dx_d^2,
\]
where $a_1,\ldots,a_d\in\F_q^\times$. Because the change of variables
is invertible, it does not change the number of zeros of $Q$. Let
\[
G=\sum_{x\in\F_q}\psi(x^2)
\]
be the standard quadratic Gauss sum. For any $a\in\F_q^\times$, the
scaling property of Gauss sums gives
\[
\sum_{x\in\F_q}\psi(ax^2)=\chi(a)G.
\]
Therefore, for fixed $t\ne0$,
\[
\sum_{x\in\F_q^d}\psi(tQ(x))
=\prod_{i=1}^d\sum_{x_i\in\F_q}\psi(ta_ix_i^2)
=\prod_{i=1}^d\chi(ta_i)G
=\chi(t)^d\chi(a_1\cdots a_d)G^d.
\]

If $d$ is odd, then $\chi(t)^d=\chi(t)$, and
\[
\sum_{t\ne0}\chi(t)=0
\]
because exactly half of the nonzero elements of $\F_q$ are squares
and half are nonsquares. Hence the total contribution of the $t\ne0$
terms vanishes, and
\[
Z(Q)=q^{d-1}.
\]

Now suppose $d$ is even. Then $\chi(t)^d=1$ for every $t\ne0$,
and
\[
\sum_{t\ne0}\chi(t)^d=\sum_{t\ne0}1=q-1.
\]
Thus
\[
Z(Q)
=q^{d-1}
+\frac1q(q-1)\chi(a_1\cdots a_d)G^d.
\]
Using the standard Gauss sum formula
\[
G^2=\chi(-1)q,
\]
we obtain
\[
G^d=(G^2)^{d/2}
=\chi(-1)^{d/2}q^{d/2}
=\chi\bigl((-1)^{d/2}\bigr)q^{d/2}.
\]
Therefore
\[
\frac{G^d}{q}
=\chi\bigl((-1)^{d/2}\bigr)q^{(d-2)/2}.
\]
Finally, since the diagonalization step changes the determinant of the
Gram matrix by a nonzero square, we have
\[
\chi(a_1\cdots a_d)=\chi(\det A).
\]
Substituting these expressions gives
\[
Z(Q)
=q^{d-1}
+(q-1)\chi\bigl((-1)^{d/2}\det A\bigr)q^{(d-2)/2}.
\]
This completes the proof.
\end{proof}

Let $m=n-1$, and fix
\[
U_0\in C_{m,q}(s),\qquad W=\Col U_0.
\]
Then $W$ is an $s$-dimensional totally isotropic subspace of the
standard dot-product space $\F_q^m$. Define
\[
E_{m,s}
=\#\{(w,c)\in W^\perp\times\F_q:w^Tw+c^2=0\}.
\]

\begin{lemma}
\label{lem:Ems}
If $q$ is odd, then
\[
\boxed{
E_{m,s}=
\begin{cases}
q^{m-s},&m\text{ even},\\[1mm]
q^{m-s}
+(q-1)\chi\!\left((-1)^{(m+1)/2}\right)q^{(m-1)/2},
&m\text{ odd}.
\end{cases}}
\]
This number depends only on $m,s,q$, and not on the particular totally
isotropic subspace $W$.
\end{lemma}

\begin{proof}
Fix $U_0\in C_{m,q}(s)$, and let $W=\Col U_0$. Then
\[
\dim W=s,
\]
and the condition $U_0^TU_0=0$ implies that $W$ is totally isotropic
with respect to the standard dot product:
\[
W\subseteq W^\perp.
\]
Here
\[
W^\perp=\{x\in\F_q^m:x^Ty=0\text{ for all }y\in W\}.
\]
Since $W\subseteq W^\perp$, the quotient space
\[
H=W^\perp/W
\]
is well defined. The dot product on $W^\perp$ induces a symmetric
bilinear form on $H$ by
\[
\bar x\cdot\bar y=x^Ty,
\]
where $x,y\in W^\perp$ are arbitrary representatives of the cosets
$\bar x,\bar y$. This is well defined because any two representatives
differ by an element of $W$, and $W$ is orthogonal to $W^\perp$.
The induced quadratic form on $H$ is
\[
Q_H(\bar x)=x^Tx,
\]
again well defined for the same reason.

We claim that $Q_H$ is nondegenerate. First, the induced bilinear form on $H$ is given by
\[
B_H(\bar x,\bar y)=x^Ty,
\]
where $x,y\in W^\perp$ are arbitrary representatives of the cosets $\bar x,\bar y$. Now suppose $\bar x\in H$ satisfies
\[
B_H(\bar x,\bar y)=0
\qquad\text{for all }\bar y\in H.
\]
By the definition of $B_H$, this means
\[
x^Ty=0
\qquad\text{for all }y\in W^\perp,
\]
where $x\in W^\perp$ is a representative of $\bar x$. Therefore
\[
x\in (W^\perp)^\perp.
\]
Since the standard dot product on $\F_q^m$ is nondegenerate, for every subspace $U\subseteq\F_q^m$ we have $(U^\perp)^\perp=U$. Taking $U=W$, we obtain
\[
(W^\perp)^\perp=W.
\]
Thus $x\in W$, which means that the coset $\bar x=x+W$ is the zero coset in $H=W^\perp/W$. Hence $\bar x=0$. Thus $H$ is a
nondegenerate quadratic space of dimension
\[
\dim H=\dim W^\perp-\dim W=(m-s)-s=m-2s.
\]

Now consider the set
\[
E_{m,s}=\#\{(w,c)\in W^\perp\times\F_q:w^Tw+c^2=0\}.
\]
For each coset $\bar w\in H$, the number of lifts $w\in W^\perp$ is
exactly $|W|=q^s$. Moreover, all lifts have the same squared length,
because if $w,w'\in W^\perp$ represent the same coset, then
$w-w'\in W$, and $W$ is totally isotropic, so
\[
w^Tw=(w'+(w-w'))^T(w'+(w-w'))
=w'^Tw'+2w'^T(w-w')+(w-w')^T(w-w')
=w'^Tw',
\]
where we used $w'^T(w-w')=0$ because $W\subseteq W^\perp$, and
$(w-w')^T(w-w')=0$ because $w-w'\in W$. Thus the squared length
depends only on the coset $\bar w$.

Therefore the number of pairs $(w,c)$ with $w$ lifting a fixed
$\bar w$ and satisfying $w^Tw+c^2=0$ is exactly the number of
$c\in\F_q$ such that
\[
Q_H(\bar w)+c^2=0.
\]
This is the same as counting zeros of the quadratic form
\[
Q_H(\bar w)+c^2
\]
on the space $H\bot\langle1\rangle$, where the added one-dimensional
form is $c^2$. Hence
\[
E_{m,s}=q^s\, Z\bigl(H\bot\langle1\rangle\bigr),
\]
where $Z(Q)$ denotes the number of zeros of a quadratic form $Q$.

We need the determinant, up to squares, of the Gram matrix of the
quadratic form on $H\bot\langle1\rangle$. Let
\[
d=\dim(H\bot\langle1\rangle)=\dim H+1=m-2s+1.
\]
The standard dot product on $\F_q^m$ has Gram determinant $1$.
We now justify the required hyperbolic splitting.  Choose a basis
$w_1,\ldots,w_s$ of $W$.  The linear functionals
$x\mapsto w_i^Tx$ are linearly independent because the ambient dot
product is nondegenerate.  Hence there exist vectors
$z_1,\ldots,z_s\in\F_q^m$ such that
\[
 w_i^Tz_j=\delta_{ij}.
\]
Let $M=(z_i^Tz_j)_{i,j}$ and replace each $z_i$ by
\[
 z_i'=z_i-\frac12\sum_{j=1}^s M_{ij}w_j.
\]
This replacement is legitimate because $q$ is odd.  Since
$w_k^Tw_j=0$, it preserves the pairings $w_k^Tz_i'=\delta_{ki}$.
Moreover, using the symmetry of $M$ and the total isotropy of $W$, we
obtain for all $i,\ell$
\[
 \begin{aligned}
 (z_i')^Tz_\ell'
 &=M_{i\ell}-\frac12M_{i\ell}-\frac12M_{\ell i}+0\\
 &=0.
 \end{aligned}
\]
Thus each pair $(w_i,z_i')$ spans a hyperbolic plane, and distinct
pairs are mutually orthogonal.  Put
\[
 Z=\Span\{z_1',\ldots,z_s'\},
 \qquad K=(W\mathbin\oplus Z)^\perp.
\]
The Gram matrix on the span of the $w_i$ and $z_i'$ is
$\begin{psmallmatrix}0&I_s\\I_s&0\end{psmallmatrix}$ and is
nondegenerate.  It follows that
\[
 \F_q^m\cong\mathbb H^s\mathbin\bot K.
\]
In this decomposition, a vector is orthogonal to every $w_i$ exactly
when its $Z$-component vanishes.  Therefore
$W^\perp=W\mathbin\oplus K$, and the quotient form on
$H=W^\perp/W$ is isometric to the form on $K$.  We may consequently
write the preceding decomposition as
\[
 \F_q^m\cong\mathbb H^s\mathbin\bot H.
\]
Each hyperbolic plane has Gram determinant $-1$.
Taking determinants up to squares gives
\[
1\equiv (-1)^s\det H \pmod{(\F_q^\times)^2},
\]
because the standard Gram determinant is $1$. Hence
\[
\det H\equiv (-1)^s.
\]
Adding the one-dimensional form $\langle1\rangle$ does not change the
square class of the determinant, so
\[
\det(H\bot\langle1\rangle)\equiv \det H\equiv (-1)^s.
\]

By Lemma $\ref{lem:quadratic-zeros}$, the number of zeros of a nondegenerate
quadratic form of dimension $d$ and determinant $\Delta$ is
\[
Z=
\begin{cases}
q^{d-1},&d\text{ odd},\\[1mm]
q^{d-1}+(q-1)\chi\bigl((-1)^{d/2}\Delta\bigr)q^{(d-2)/2},
&d\text{ even}.
\end{cases}
\]
Here $d=m-2s+1$ and $\Delta\equiv(-1)^s$.

\paragraph{Case 1: $m$ even.}
Then $d=m-2s+1$ is odd. Hence
\[
Z\bigl(H\bot\langle1\rangle\bigr)=q^{d-1}=q^{m-2s}.
\]
Therefore
\[
E_{m,s}=q^s q^{m-2s}=q^{m-s}.
\]

\paragraph{Case 2: $m$ odd.}
Then $d=m-2s+1$ is even. We compute the quadratic character factor:
\[
\begin{aligned}
\chi\bigl((-1)^{d/2}\Delta\bigr)
&=\chi\bigl((-1)^{(m-2s+1)/2}(-1)^s\bigr)\\
&=\chi\bigl((-1)^{(m-2s+1)/2+s}\bigr)\\
&=\chi\bigl((-1)^{(m+1)/2}\bigr).
\end{aligned}
\]
Thus
\[
Z\bigl(H\bot\langle1\rangle\bigr)
=q^{d-1}
+(q-1)\chi\bigl((-1)^{(m+1)/2}\bigr)q^{(d-2)/2}.
\]
Now
\[
d-1=m-2s,
\qquad
\frac{d-2}{2}=\frac{m-2s-1}{2}.
\]
Multiplying by $q^s$, we obtain
\[
\begin{aligned}
E_{m,s}
&=q^s q^{m-2s}
+q^s(q-1)\chi\bigl((-1)^{(m+1)/2}\bigr)q^{(m-2s-1)/2}\\
&=q^{m-s}
+(q-1)\chi\bigl((-1)^{(m+1)/2}\bigr)q^{s+(m-2s-1)/2}\\
&=q^{m-s}
+(q-1)\chi\bigl((-1)^{(m+1)/2}\bigr)q^{(m-1)/2}.
\end{aligned}
\]

This proves both cases. Since the argument depends only on $m,s,q$
and not on the particular choice of the totally isotropic subspace
$W$, the number $E_{m,s}$ is independent of $W$.
\end{proof}
\subsection{Exact Recurrence in Odd Characteristic}
For odd $q$, define the parity-dependent correction
\[
 \kappa_n(q)=
 \begin{cases}
 (q-1)\chi\!\left((-1)^{n/2}\right)q^{(n-2)/2},&n\text{ even},\\[1mm]
 0,&n\text{ odd}.
 \end{cases}
\]
Both exponents in the first line are integers because that line is used
only when $n$ is even.  The following theorem gives the exact
rank-refined recurrence for the number of upper triangular Cholesky
roots of the zero matrix over any finite field of odd characteristic.
\begin{theorem}
\label{thm:odd-recurrence}
Let $q$ be odd. Then
\[
 c_{n,0}^{(q)}=1,
 \qquad
 c_{n,r}^{(q)}=0\quad\text{if }r>\left\lfloor\frac n2\right\rfloor.
\]
For $1\le r\le\lfloor n/2\rfloor$,
\[
\boxed{
\begin{aligned}
c_{n,r}^{(q)}
={}&q^r c_{n-1,r}^{(q)}+\left[
q^{n-r}-q^{r-1}+\kappa_n(q)
\right]c_{n-1,r-1}^{(q)}.
\end{aligned}}
\]
The total number is
\[
|C_{n,q}|=\sum_{r\ge0}c_{n,r}^{(q)}.
\]
\end{theorem}

\begin{proof}
The rank-zero assertion follows because the zero matrix is the only
matrix of rank zero and it satisfies $0^T0=0$.  For the support bound,
let $U^TU=0$ and put $W=\Col U$.  For any $Ux,Uy\in W$,
\[
 (Ux)^T(Uy)=x^TU^TUy=0,
\]
so $W\subseteq W^\perp$.  Nondegeneracy of the standard dot product
gives $\dim W^\perp=n-\dim W$; therefore
$2\rank U\le n$.  It remains to prove the recurrence for
$1\le r\le\lfloor n/2\rfloor$.

We analyze the possible one-step extensions of an upper triangular matrix satisfying $U^TU=0$. Write
\[
U = \begin{pmatrix} U_0 & w \\ 0 & c \end{pmatrix},
\]
where $U_0\in\T_{n-1}(\F_q)$, $w\in\F_q^{n-1}$, and $c\in\F_q$. Then
\[
U^TU
=
\begin{pmatrix}
U_0^T & 0\\
w^T & c
\end{pmatrix}
\begin{pmatrix}
U_0 & w\\
0 & c
\end{pmatrix}
=
\begin{pmatrix}
U_0^TU_0 & U_0^Tw\\
w^TU_0 & w^Tw+c^2
\end{pmatrix}.
\]
Thus $U^TU=0$ is equivalent to the three conditions
\[
U_0^TU_0=0,\qquad U_0^Tw=0,\qquad w^Tw+c^2=0.
\]
The first condition says $U_0\in C_{n-1,q}$. Let
\[
W=\Col U_0,\qquad s=\rank U_0.
\]
The second condition $U_0^Tw=0$ is equivalent to $w\in W^\perp$.
Because $U_0^TU_0=0$, the column space $W$ is totally isotropic,
that is, $W\subseteq W^\perp$. Therefore, for a fixed old root
$U_0$ of rank $s$, the valid pairs $(w,c)$ are exactly those in
the set
\[
\{(w,c)\in W^\perp\times\F_q : w^Tw+c^2=0\},
\]
whose cardinality is $E_{n-1,s}$ by definition, with $m=n-1$.
By Lemma~\ref{lem:Ems}, this number depends only on $n-1$, $s$,
and $q$.

We now determine the rank of the extended matrix $U$. The first
$n-1$ columns of $U$ are the columns of
\[
\begin{pmatrix} U_0 \\ 0 \end{pmatrix},
\]
whose column space has dimension $s$. If $w\in W$ and $c=0$,
then the last column
\[
\begin{pmatrix} w \\ 0 \end{pmatrix}
\]
lies in this column space, so the rank remains $s$. Since $W$ is
totally isotropic, every $w\in W$ satisfies $w^Tw=0$, and the
equation $w^Tw+c^2=0$ forces $c=0$. Hence for each $w\in W$
there is exactly one valid pair $(w,0)$, and the rank is preserved.
Conversely, if $w\notin W$ or $c\neq 0$, the last column is not in
the span of the previous columns, so the rank increases to $s+1$.

Thus, among the $E_{n-1,s}$ valid extensions of a fixed old root
$U_0$ of rank $s$, exactly $|W|=q^s$ preserve the rank, and the
remaining $E_{n-1,s}-q^s$ increase the rank by one.

Now we count matrices of rank $r$ in $C_{n,q}$. Such a matrix
arises either from an old root of rank $r$ with a rank-preserving
extension, or from an old root of rank $r-1$ with a rank-increasing
extension. Therefore
\[
c_{n,r}^{(q)}
=
q^r c_{n-1,r}^{(q)}
+
\bigl( E_{n-1,r-1} - q^{r-1} \bigr) c_{n-1,r-1}^{(q)}.
\]
Here $c_{n-1,r-1}^{(q)}$ is the number of old roots of rank $r-1$.

It remains to evaluate $E_{n-1,r-1}$ using Lemma~\ref{lem:Ems} with
$m=n-1$ and $s=r-1$.

If $n$ is odd, then $m=n-1$ is even, so the lemma gives
\[
E_{n-1,r-1}
=
q^{(n-1)-(r-1)}
=
q^{n-r}.
\]
Hence
\[
E_{n-1,r-1}-q^{r-1}
=
q^{n-r}-q^{r-1},
\]
and the recurrence becomes
\[
c_{n,r}^{(q)}
=
q^r c_{n-1,r}^{(q)}
+
\bigl( q^{n-r}-q^{r-1} \bigr) c_{n-1,r-1}^{(q)}.
\]

If $n$ is even, then $m=n-1$ is odd, so the lemma gives
\[
E_{n-1,r-1}
=
q^{n-r}
+
(q-1)\chi\bigl((-1)^{n/2}\bigr) q^{(n-2)/2}.
\]
Therefore
\[
E_{n-1,r-1}-q^{r-1}
=
q^{n-r}-q^{r-1}
+
(q-1)\chi\bigl((-1)^{n/2}\bigr) q^{(n-2)/2}.
\]

Combining both cases and using the definition of $\kappa_n(q)$ yields
the stated formula:
\[
c_{n,r}^{(q)}
=
q^r c_{n-1,r}^{(q)}
+
\left[
q^{n-r}-q^{r-1}+\kappa_n(q)
\right]
c_{n-1,r-1}^{(q)}.
\]

Finally, the total number of Cholesky roots of the zero matrix is obtained by
summing over all possible ranks:
\[
|C_{n,q}|=\sum_{r\ge 0}c_{n,r}^{(q)}.
\]
This completes the proof.
\end{proof}
\begin{corollary}
\label{cor:n2-odd}
If $q$ is odd, then
\[
|C_{2,q}|
=1+(q-1)(1+\chi(-1))
=
\begin{cases}
2q-1,&\chi(-1)=1,\\
1,&\chi(-1)=-1.
\end{cases}
\]
On the other hand,
\[
|B_{2,q}|=q.
\]
Thus in odd characteristic we generally no longer have
$|B_{n,q}|=|C_{n,q}|$.
\end{corollary}

\begin{proof}
Let $U\in\T_2(\F_q)$ be arbitrary, and write
\[
U=
\begin{pmatrix}
a & b\\
0 & c
\end{pmatrix},
\qquad a,b,c\in\F_q.
\]
Then
\[
U^TU
=
\begin{pmatrix}
a & 0\\
b & c
\end{pmatrix}
\begin{pmatrix}
a & b\\
0 & c
\end{pmatrix}
=
\begin{pmatrix}
a^2 & ab\\
ab & b^2+c^2
\end{pmatrix}.
\]
Thus $U^TU=0$ is equivalent to
\[
a^2=0,\qquad ab=0,\qquad b^2+c^2=0.
\]
Since $q$ is odd, $\F_q$ is a field of characteristic not $2$, and
the only solution of $a^2=0$ is $a=0$. The condition $ab=0$ is
then automatic. Hence $U\in C_{2,q}$ if and only if
\[
a=0
\qquad\text{and}\qquad
b^2+c^2=0.
\]
So every element of $C_{2,q}$ has the form
\[
U=
\begin{pmatrix}
0 & b\\
0 & c
\end{pmatrix},
\qquad b^2+c^2=0.
\]

We now count the pairs $(b,c)\in\F_q^2$ satisfying
\[
b^2+c^2=0.
\]
If $b=0$, then $c^2=0$, so $c=0$. This gives exactly one solution,
namely $(0,0)$.

If $b\neq 0$, then $b$ is invertible. Put $t=cb^{-1}$. Then
\[
b^2+c^2=b^2\bigl(1+t^2\bigr)=0,
\]
and since $b^2\neq0$, this is equivalent to
\[
t^2=-1.
\]
For each fixed $t\in\F_q$ satisfying $t^2=-1$, the corresponding
pairs are
\[
(b,c)=(b,tb),
\qquad b\in\F_q^\times,
\]
of which there are $q-1$. Therefore the number of solutions with $b\neq 0$ is
\[
(q-1)\cdot \#\{t\in\F_q:t^2=-1\}.
\]
Over a finite field of odd characteristic, the equation $t^2=-1$ has
two solutions if $-1$ is a square, and no solutions if $-1$ is a
nonsquare. In terms of the quadratic character $\chi$, this count is
\[
\#\{t\in\F_q:t^2=-1\}=1+\chi(-1),
\]
where $\chi(-1)=1$ when $-1$ is a square and $\chi(-1)=-1$
otherwise. Hence the total number of solutions is
\[
1+(q-1)\bigl(1+\chi(-1)\bigr).
\]
If $\chi(-1)=1$, this gives
\[
1+2(q-1)=2q-1.
\]
If $\chi(-1)=-1$, it gives
\[
1+0\cdot(q-1)=1.
\]
Therefore
\[
|C_{2,q}|
=
1+(q-1)\bigl(1+\chi(-1)\bigr)
=
\begin{cases}
2q-1,&\chi(-1)=1,\\
1,&\chi(-1)=-1.
\end{cases}
\]

Let $X\in\T_2(\F_q)$ be arbitrary, and write
\[
X=
\begin{pmatrix}
a & b\\
0 & c
\end{pmatrix}.
\]
Then
\[
X^2=
\begin{pmatrix}
a^2 & b(a+c)\\
0 & c^2
\end{pmatrix}.
\]
The condition $X^2=0$ is equivalent to
\[
a^2=0,\qquad c^2=0,\qquad b(a+c)=0.
\]
Since $q$ is odd, $a^2=0$ implies $a=0$, and $c^2=0$ implies
$c=0$. The remaining condition $b(a+c)=0$ is then automatic. Thus
$X\in B_{2,q}$ if and only if
\[
a=0,\qquad c=0,
\]
while $b\in\F_q$ is arbitrary. Consequently every element of
$B_{2,q}$ has the form
\[
X=
\begin{pmatrix}
0 & b\\
0 & 0
\end{pmatrix},
\qquad b\in\F_q,
\]
and there are exactly $q$ such matrices. Hence
\[
|B_{2,q}|=q.
\]

For every odd $q$, we have $q\ge 3$. If $\chi(-1)=-1$, then
\[
|C_{2,q}|=1\neq q=|B_{2,q}|.
\]
If $\chi(-1)=1$, then
\[
|C_{2,q}|=2q-1\neq q=|B_{2,q}|,
\]
because $2q-1=q$ would force $q=1$, contradicting $q\ge3$.
Thus in odd characteristic we generally have
\[
|C_{2,q}|\neq |B_{2,q}|,
\]
and therefore the total equinumerosity $|B_{n,q}|=|C_{n,q}|$ fails
already for $n=2$.
\end{proof}
\begin{example}
Over $\F_3$,
\[
|C_{1,3}|=1,\qquad
|C_{2,3}|=1,\qquad
|C_{3,3}|=9,\qquad
|C_{4,3}|=153.
\]
More precisely,
\[
c_{4,0}^{(3)}=1,\qquad
c_{4,1}^{(3)}=56,\qquad
c_{4,2}^{(3)}=96.
\]
\end{example}

\begin{example}
Over $\F_5$,
\[
|C_{1,5}|=1,\qquad
|C_{2,5}|=9,\qquad
|C_{3,5}|=65,
\]
where
\[
c_{3,0}^{(5)}=1,\qquad c_{3,1}^{(5)}=64.
\]
\end{example}

\begin{remark}
When $q$ is even, the recurrence matches the square-zero upper
triangular matrix count of Ekhad--Zeilberger
\cite{EkhadZeilberger1996}. When $q$ is odd, the
deviation appears only in even-order extension steps, and its sign is
governed by
\[
\chi((-1)^{n/2}).
\]
This shows that the characteristic-two phenomenon observed in the
original paper is not accidental; it reflects the structural difference
between the bijectivity of the square map in even characteristic and the
zero-count deviation of quadratic forms in odd characteristic.
\end{remark}

\section{Computational Verification}
\label{sec:verification}

\subsection{Verification Scope}

The formulas in this paper have been independently checked by
exhaustive enumeration as follows.
\begin{enumerate}[label=\textup{(\arabic*)}]
\item For every symmetric target, \Cref{thm:general-recursion} agrees
      with direct enumeration over $\F_2$ for $n\le4$, over $\F_3$
      for $n\le3$, and over $\F_5$ for $n\le2$.
\item Over $\F_2$, for every binary diagonal target with $n\le5$,
      both the total counts and the rank distributions from
      \Cref{thm:diagonal-recurrence} agree with exhaustive enumeration.
\item The zero-fiber rank recurrences agree with direct enumeration
      over $\F_3$ for $n\le4$ and over $\F_5$ and $\F_7$ for
      $n\le3$.
\end{enumerate}

In particular, the first five total numbers of Cholesky roots of the
zero matrix over $\F_2$ are
\[
1,\ 2,\ 6,\ 28,\ 192
\qquad(n=1,2,3,4,5),
\]
in agreement with \cite{CooperWhitlatch2025}.  The implementations used
for these checks are collected in \Cref{app:algorithms}.  The
arbitrary-target routine is written for prime fields; over a nonprime
field its integer modular arithmetic must be replaced by arithmetic in a
chosen finite-field representation.  The binary diagonal routine needs
no such representation, and the zero-fiber routine depends only on the
field order $q$.

\section{Conclusion}

The principal contribution of this paper is a fiber-level theory for the
singular part of the triangular Cholesky map.  The explicit construction
in \Cref{thm:explicit-bijection} upgrades the previously known binary
rank-refined equinumerosity to an invertible rank-preserving bijection.
The recurrences in \Cref{thm:even-recurrence,thm:odd-recurrence} then
determine the entire rank distribution of the zero fiber over every
finite field: even characteristic preserves the square-zero recurrence,
while odd characteristic contributes a correction governed by the
discriminant of a quotient quadratic space.

The general recursion of \Cref{thm:general-recursion} places these
zero-target results within an exact theory for every prescribed symmetric
target.  Its specialization to the LPM locus recovers the constant regular
fiber sizes compatible with the generalized Cholesky parametrizations of
\cite{Vishwakarma2027}; zero pivots are precisely where rank-one branching
begins.  For binary diagonal targets, \Cref{thm:diagonal-recurrence}
compresses that branching into an $O(n^2)$ arithmetic state recursion.
Thus the target-class enumerations of Vishwakarma and of Ayyer--Prasad
\cite{AyyerPrasad2026} and the fixed-target multiplicities developed here
address complementary levels of the finite-field Cholesky problem.
Together, the results give complete recursive answers to the three
questions of \cite{CooperWhitlatch2025}.

Building on these results, several natural problems remain open for future investigation.

\begin{enumerate}[label=\textup{(\roman*)}]
\item Solve the odd-characteristic recurrence to obtain a closed summation formula and asymptotic estimates analogous to the Ekhad--Zeilberger formula.
\item Classify the orbits of the invertible upper triangular congruence action
      \[
      A\longmapsto P^TAP
      \]
      on $\Sym_n(\F_q)$, and give $\nu_q(A)$ directly for each normal
      form, including the singular boundary orbits.
\item Determine whether the branching states in the general matrix
      recursion can be compressed by finitely many flag invariants, so
      that the worst-case exponential algorithm can be improved to a
      parameterized polynomial algorithm.
\item Over $\F_{2^e}$, find a version of the rank-preserving bijection
      that is Frobenius equivariant and avoids the RREF--lexicographic
      choice.  This would strengthen the present flag-compatible
      construction without changing its rank behavior.
\item For a nondegenerate symmetric form $G$, study the twisted zero
      fibers $\{U\in\T_n(\F_q):U^TGU=0\}$ and determine how their rank
      recurrences depend on the Witt type of $G$.
\item Study the Hermitian analogue
      \[
      U^\ast U=0
      \]
      in odd characteristic, and its relation to unitary quadratic spaces and the enumeration of Hermitian totally isotropic flags.
\end{enumerate}

The regular--singular distinction developed here provides a basis for
these extensions of finite-field Cholesky fiber theory.

\appendix

\section{Algorithms and Reproducible Recurrences}
\label{app:algorithms}

This appendix collects all program listings used in the paper.  The
first routine uses integer representatives and is therefore written for
prime fields.  The second routine is specifically binary.  The third
depends only on the field order and works for arbitrary prime powers.
All matrices passed to the first routine are immutable tuples of tuples,
which permits exact memoization.

\subsection{The arbitrary-target recursion over a prime field}
\label{app:general-code}

The next routine is a direct implementation of
\Cref{thm:general-recursion}.  Its return value is an exact integer; no
random choices are made.

\begin{lstlisting}[style=pythonstyle,basicstyle=\ttfamily\footnotesize,
caption={Cholesky fiber cardinality over a prime field}]
from functools import lru_cache
from itertools import product

def subtract_outer(D, x, q, scale=1):
    """Return D - scale*x*x^T over the prime field F_q."""
    m = len(D)
    return tuple(
        tuple((D[i][j] - scale * x[i] * x[j]) % q
              for j in range(m))
        for i in range(m)
    )

@lru_cache(None)
def cholesky_count_prime(A, q):
    """Number of upper triangular U over F_q with U^T U = A."""
    n = len(A)
    if n == 0:
        return 1

    if any(A[i][j] % q != A[j][i] % q
           for i in range(n) for j in range(n)):
        return 0

    a = A[0][0] % q
    b = tuple(A[0][j] % q for j in range(1, n))
    D = tuple(
        tuple(A[i][j] % q for j in range(1, n))
        for i in range(1, n)
    )

    if a != 0:
        roots = [u for u in range(q) if (u * u) % q == a]
        if not roots:
            return 0
        inverse_a = pow(a, -1, q)
        target = subtract_outer(D, b, q, scale=inverse_a)
        return len(roots) * cholesky_count_prime(target, q)

    if any(b):
        return 0

    total = 0
    for x in product(range(q), repeat=n - 1):
        target = subtract_outer(D, x, q)
        total += cholesky_count_prime(target, q)
    return total
\end{lstlisting}

For $q=2$, subtraction agrees with addition and the routine specializes
to \Cref{cor:binary-recursion}.  In the worst case it still enumerates
$q^{n-1}$ rank-one perturbations at a zero leading row; the cache avoids
recomputing identical smaller targets.

\subsection{Rank distribution for a binary diagonal target}
\label{app:diagonal-code}

The dictionary key in the internal state is the excess rank $t$, not
the absolute rank.  The assertion on the coefficient is a runtime check
of the admissible-state bound proved in
\Cref{thm:diagonal-recurrence}.

\begin{lstlisting}[style=pythonstyle,basicstyle=\ttfamily\footnotesize,
caption={Rank distribution for binary diagonal targets}]
def diagonal_rank_distribution(bits):
    """
    bits = (d_1,...,d_n), with each d_i in {0,1}.
    Return {root_rank: number_of_roots}.
    """
    h = {0: 1}  # key: excess rank t
    s = 0       # number of ones already read
    z = 0       # number of zeros already read

    for d in bits:
        new = {}
        if d == 1:
            s += 1
            for t, value in h.items():
                coefficient = 2 ** (z - t)
                new[t] = new.get(t, 0) + coefficient * value
        else:
            z += 1
            for t, value in h.items():
                # Preserve the excess rank t.
                new[t] = new.get(t, 0) + (2 ** t) * value

                # Increase the excess rank from t to t+1.
                coefficient = 2 ** (z - t - 1) - 2 ** t
                assert coefficient >= 0
                if coefficient:
                    new[t + 1] = (
                        new.get(t + 1, 0) + coefficient * value
                    )

        h = {t: value for t, value in new.items() if value}

    return {s + t: value for t, value in h.items()}
\end{lstlisting}

\subsection{Rank distributions in the zero fiber}
\label{app:zero-code}

The correction term in odd characteristic requires only the quadratic
character of $-1$, which is determined by $q\bmod4$ for every odd
prime power.  Thus the routine below takes the field order $q$ as an
integer and implements both \Cref{thm:even-recurrence} and
\Cref{thm:odd-recurrence}.

\begin{lstlisting}[style=pythonstyle,basicstyle=\ttfamily\footnotesize,
caption={Rank recurrences for Cholesky roots of the zero matrix}]
def zero_cholesky_rank_counts(n_max, q):
    """
    Return rank distributions for orders 0,...,n_max.
    The integer q must be a prime power.
    """
    if q < 2:
        raise ValueError("q must be a prime power")

    even_characteristic = (q % 2 == 0)
    chi_minus_one = None
    if not even_characteristic:
        chi_minus_one = 1 if q % 4 == 1 else -1

    distributions = [{0: 1}]
    previous = {0: 1}

    for n in range(1, n_max + 1):
        current = {}
        for old_rank, value in previous.items():
            # Preserve the old rank.
            current[old_rank] = (
                current.get(old_rank, 0)
                + (q ** old_rank) * value
            )

            # Increase the old rank by one.
            new_rank = old_rank + 1
            coefficient = q ** (n - new_rank) - q ** old_rank

            if (not even_characteristic) and n % 2 == 0:
                exponent = n // 2
                chi_power = 1 if exponent % 2 == 0 else chi_minus_one
                coefficient += (
                    (q - 1) * chi_power * q ** ((n - 2) // 2)
                )

            assert coefficient >= 0
            if coefficient:
                current[new_rank] = (
                    current.get(new_rank, 0)
                    + coefficient * value
                )

        current = {r: value for r, value in current.items() if value}
        distributions.append(current)
        previous = current

    return distributions
\end{lstlisting}


\section*{Data availability}
Data sharing not applicable to this article as no data sets were generated or analysed during the current study.

\section*{Declaration of competing interest}
The authors declare that we have no known competing financial interests or personal relationships that 
could have appeared to influence the work reported in this paper.

\end{document}